\documentclass[11pt]{article}
\usepackage[a4paper,centering,scale=0.85]{geometry}

\usepackage{stix}
\usepackage{xcolor}
\usepackage{graphicx}
\usepackage{amsmath,amsthm}
\usepackage{amsfonts}
\usepackage{fixmath}

\newtheorem{theorem}{Theorem}[section]
\newtheorem{lemma}{Lemma}[section]

\theoremstyle{definition}
\newtheorem{remark}{Remark}[section]
\newtheorem{example}{Example}[section]
\usepackage{hyperref}
\hypersetup{colorlinks}

\usepackage{algorithm,algorithmicx,algpseudocode}
\usepackage{subfig}
\def\wtd{\widetilde}
\def\what{\widehat}
\DeclareMathOperator{\opt}{opt}
\DeclareMathOperator{\rank}{rank}
\DeclareMathOperator{\subspan}{span}
\DeclareMathOperator{\diag}{diag}
\DeclareMathOperator{\HH}{H}

\def\range{{\cal R}}
\def\krylov{{\cal K}}
\def\interval{{\cal I}}
\def\domain{{\cal D}}

\DeclareMathOperator{\eSD}{o}
\DeclareMathOperator{\SD}{SD}
\DeclareMathOperator{\LOCG}{LOCG}
\DeclareMathOperator{\inertia}{inertia}
\DeclareMathOperator{\pinv}{\dagger}
\DeclareMathOperator{\diff}{d\!}

\numberwithin{algorithm}{section}

\allowdisplaybreaks

\title{Convergence Analysis of Extended LOBPCG for Computing Extreme Eigenvalues}

\author{
	Peter Benner\thanks{
		Max Planck Institute for Dynamics of Complex Technical Systems,
		Sandtorstra{\ss}e 1, 39106 Magdeburg, Germany.
		E-mail: {\tt benner@mpi-magdeburg.mpg.de}.
	}
	\and
	Xin Liang\thanks{
		Yau Mathematical Sciences Center,
		Tsinghua University,
		Beijing 100084, China.
		E-mail: {\tt liangxinslm@tsinghua.edu.cn}.
	}
}
\date{\today}

\begin{document}

\maketitle

\begin{abstract}
	This paper is concerned with the convergence analysis of an extended variation of the locally optimal preconditioned conjugate gradient method (LOBPCG)
	for the extreme eigenvalue of a Hermitian matrix polynomial which admits some extended form of Rayleigh quotient.
	This work is a generalization of the analysis by Ovtchinnikov ({\em SIAM J. Numer. Anal.}, 46(5):2567--2592, 2008).
	As instances, the algorithms for definite matrix pairs and hyperbolic quadratic matrix polynomials are shown to be globally convergent and to have an asymptotically local convergence rate.
	Also, numerical examples are given to illustrate the convergence.
\end{abstract}



\section{Introduction}\label{sec:intro}
Given a Hermitian matrix polynomial
\begin{equation}\label{eq:Fpoly}
	F(\lambda)=\sum_{k=0}^mA_k\lambda^{m-k},
\end{equation}
of degree $m$, where $A_k\in\mathbb{C}^{n\times n}$ for $k=0,\ldots,m$,
and an interval $\interval=(\lambda_-,\lambda_+)$.
Suppose $F(\lambda_-)$ is negative definite.
For some nonzero $x\in\mathbb{C}^n$,
consider the equation
\begin{equation}\label{eq:quadrik}
x^{\HH}F(\lambda)x=0, \quad \lambda\in\interval.
\end{equation}
Let $\domain$ denote the set of all $x$ for which \eqref{eq:quadrik} has at lease one root $\lambda=\rho(x)$ in $\interval$,
while $\mathbb{C}^n\setminus\domain$ is the set of all $x$ for which \eqref{eq:quadrik} has no root in $\interval$.
For any $x\in \domain$, define
\[
	\sigma(x):=x^{\HH}F'(\rho(x))x=\sum_{k=0}^{m-1}(m-k)A_k\rho(x)^{m-k-1}.
\]
Suppose that $\sigma(x)>0$ for any $x\in\domain$.
Then \eqref{eq:quadrik} has only one root in $\interval$,
which is called the Rayleigh quotient of $F(\lambda)$ at $x$.
Suppose the matrix polynomial $F(\lambda)$ has $\ell$ eigenvalues in $\interval$, namely $\lambda_1\le\dots\le\lambda_\ell$,
while for any matrix $X\in\mathbb{C}^{n\times k}$ with a proper constraint, the projected polynomial $X^{\HH}F(\lambda)X$ has $\ell_X\le\ell$ eigenvalues in $\interval$, namely $\lambda_{1,X}\le\dots\le\lambda_{\ell_X,X}$.
Furthermore, suppose the eigenvalues of $F(\lambda)$ admit min-max principles, such as
\begin{enumerate}
	\item the \emph{Wielandt-Lidskii min-max principle}:
		\begin{equation} \label{eq:wielandt:minmax}
			\min_{
				\begin{subarray}{c}
					{\cal X}_1\subset\dots\subset{\cal X}_{k} \\
					\dim {\cal X}_j=i_j
				\end{subarray}
			}\,\,
			\max_{
				\begin{subarray}{c}
					x_j\in{\cal X}_j\\
					X=[x_1,\dots,x_{k}] \\
					\rank(X)=k\\
					\text{proper $X$}
				\end{subarray}
			}
			\sum_{j=1}^k\lambda_{j,X}=\sum_{j=1}^k\lambda_{i_j};
		\end{equation}
	\item the \emph{Courant-Fischer min-max principle} obtained by setting $k=1$ in \eqref{eq:wielandt:minmax} and noticing $\rho(x)=\lambda_{1,X}$:
		\begin{equation} \label{eq:courant-fischer:minmax}
			\min_{
				\begin{subarray}{c}
					\dim {\cal X}=i
				\end{subarray}
			}\,\,
			\max_{
				\begin{subarray}{c}
					x\in{\cal X}\\
					\text{proper $x$}
				\end{subarray}
			}
			\rho(x)=\lambda_{i};
		\end{equation}
	\item the \emph{Fan trace min principle} obtained by setting $i_j=j$ in \eqref{eq:wielandt:minmax}:
		\begin{equation}\label{eq:trace:min}
			\min_{
				\begin{subarray}{c}
					\rank(X)=k\\
					\text{proper $X$}
				\end{subarray}
			}\sum_{j=1}^k\lambda_{j,X}= \sum_{j=1}^k\lambda_j; 	
		\end{equation}
	\item or the extreme eigenvalue characterization obtained by setting $i=1$ in \eqref{eq:courant-fischer:minmax} or $k=1$ in \eqref{eq:trace:min}:
		\begin{equation}\label{eq:rayleighquotient:min}
			\min_{
				\begin{subarray}{c}
					\text{proper $x$}
				\end{subarray}
			}\rho(x)= \lambda_1, 	
		\end{equation}
\end{enumerate}
where the phrase ``proper $X$'' in the min/max means that the minimum/maximum is obtained under some proper constraint.

These min-max principles motivate us to use the Rayleigh-Ritz procedure and gradient-type optimization methods,
such as the steepest descent method (SD) or the conjugate gradient method (CG), to obtain several smallest eigenvalues and their corresponding eigenvectors.
In this view, the \emph{locally optimal block preconditioned (extended) conjugate gradient method} (LOBP(e)CG) has been developed to solve some kinds of eigenvalue problems.
Locally optimal CG for nonlinear optimization was first described by Takahashi \cite{taka:65}.
Later, Knyazev \cite{knya:01} established LOBPCG for the generalized Hermitian eigenvalue problem $A-\lambda B$, where $A\succ0$.
Because of its efficiency, this method has been used to solve different kinds of eigenvalue problems.
Nevertheless, up to now, the convergence analysis of this method has been incomplete.
As far as we know, current results on the estimate for the convergence rate fall into two categories.
Ovtchinnikov \cite{ovtc:08:jacobi:I,ovtc:08:jacobi:II} dealt with the convergence rate of a standard form for LOBPCG applied to standard Hermitian eigenvalue problems and generalized Hermitian eigenvalue problems $(A-\lambda B)x=0$ with a positive definite $B$.
He analyzed the convergence rate of LOBPCG by constructing a relationship to SD and then bringing in the convergence rate of SD by Samokish \cite{samo:58}.
On the other hand, also for those two types of eigenvalue problems, Neymeyr and his co-authors derived the convergence rate of a special form named ``sharp estimate'' for preconditioned inverse vector iteration (PINVIT) and (preconditioned) SD in a series of works \cite{neym:01:geometric,knne:03,neoz:11:convergence,neym:12:geometric,arkn:15:convergence}.
In this paper, we will consider several instances of the generalized eigenvalue problem and try to apply the developed ideas to them for the algorithm LOBPCG for computing the extreme eigenvalue, which means the block size is $1$, or equivalently, a vector version of LOBPCG.
The problems are:
\begin{enumerate}
	\item \emph{Definite matrix pair} $F(\lambda)=\lambda B-A$, which means there exists $\lambda_0\in\mathbb{R}$ such that $F(\lambda_0)\prec0$.
		Let $\interval=(\lambda_0,+\infty)$ and $\domain=\{x\in\mathbb{C}^n : x^{\HH}Bx>0\}$, and let the proper constraint be $X^{\HH}BX=I$, satisfying the assumptions above (see, e.g.~\cite{kove:95,nave:03,lilb:13,lili:13}).
		Here, the investigated algorithm coincides with the algorithm given by Kressner et al \cite[Algorithm~1]{krps:14:indefinite}.
	\item \emph{Hyperbolic quadratic matrix polynomial} $F(\lambda)=\lambda^2A+\lambda B+C$, with $A\succ0$ and assuming there exists $\lambda_0\in\mathbb{R}$ such that $F(\lambda_0)\prec0$.
		Let $\interval=(\lambda_0,+\infty)$ and $\domain=\mathbb{C}^n$, and let the proper constraint be $\rank(X)=k$ or $X^{\HH}AX=I$, satisfying the assumptions above (see, e.g.~\cite{duff:55,mark:88,gula:05,liangL2015hyperbolic}).
		Here, the investigated algorithm coincides with the algorithm given by Liang and Li \cite[Algorithm~11.2]{liangL2015hyperbolic}.
\end{enumerate}

The rest of this paper is organized as follows.
First, some notation is introduced.
Section~2 presents the generic framework of LOBPeCG for any kind of Hermitian matrix polynomial satisfying the assumptions at the beginning of the paper, and also its convergence analysis.
Section~3 applies this convergence analysis to the two problems listed above.
In Section~4, two numerical examples are given to illustrate the convergence rate.
Some conclusions are provided in Section~5.
Appendices A and B are used to take care of detailed and difficult estimates in the proof of the convergence analysis in Section~2.

\subparagraph{Notation.}
Throughout this paper,
$I_n$ (or simply $I$ if its dimension is clear from the context) is the $n\times n$ identity matrix, and $e_j$ is its $j$th column.
$\mathbf{1}_n =\sum_{j=1}^n e_j$ (or also simply $\mathbf{1}$ if its dimension is clear from the context).
$\diag(\alpha_1,\dots,\alpha_n)$ is a diagonal matrix whose diagonal entries are $\alpha_1,\dots,\alpha_n$.
$X^{\HH}$ is the conjugate transpose of a vector or matrix $X$,
and $\|X\|$ is its spectral norm,
$X^{\pinv}$ is the Moore-Penrose inverse of a matrix $X$.

Given a matrix $A$ and a vector $x$, the $(m+1)$-dimensional Krylov subspace is denoted by $\krylov_m(A,x)=\subspan\{x,Ax,\dots,A^mx\}$,
and $\range(A)$ denotes $A$'s column subspace.

We use $A\succ 0$ ($A\succeq 0$) to indicate that $A$ is Hermitian positive (semi-)definite, and $A\prec 0$ ($A\preceq 0$) if
$-A\succ 0$ ($-A\succeq 0$). For $A\succeq 0$, $A^{1/2}$ is the unique positive semidefinite square root of $A$.

For a Hermitian matrix $A$, its eigenvalues are denoted by
\[
	\lambda_{\min}(A)\le\lambda_{\min}^{(2)}(A)\le\dots\le\lambda_{\min}^{(n)}(A), \quad\text{or}\quad
	\lambda_{\max}^{(n)}(A)\le\dots\le\lambda_{\max}^{(2)}(A)\le\lambda_{\max}(A).
\]
For any two functions $f(x),g(x)$, by $f(x)\sim g(x)$ we denote the case that
$\tau_2 f(x) \le g(x) \le \tau_1 f(x) $ for some $\tau_1,\tau_2>0$ and all $x$ in the joint domain of $f$ and $g$.
Similarly, by $f_i\sim g_i$ we denote the same situation for two sequences $\{f_i\},\{g_i\}$.
Clearly ``$\sim$'' is an equivalence relation.

Recall the matrix polynomial $F(\lambda)$ from \eqref{eq:Fpoly}.
Define the corresponding residual vector $r(x):=F(\rho(x))x$.
Then $x^{\HH}r(x)=0$ and
\[
	r(x)= -\frac{1}{2}\sigma(x)\nabla\rho(x),
\]
because
\[
	0=\nabla(x^{\HH}F(\rho(x))x)=x^{\HH}F'(\rho(x))x\nabla\rho(x) + 2F(\rho(x))x
	=\sigma(x)\nabla\rho(x) + 2r(x).
\]
Denote the divided difference by
\[
	\Phi (\rho_1,\rho_2):=\frac{F(\rho_1)-F(\rho_2)}{\rho_1-\rho_2}.
\]
Then for any nonzero $x$, define
\[
	P_{x,\rho_1,\rho_2} := \frac{xx^{\HH}\Phi(\rho_1,\rho_2) }{x^{\HH}\Phi(\rho_1,\rho_2) x},
\]
and
\[
	\check F_{\rho_1,\rho_2}(\rho;x) := \left(I - P_{x,\rho_1,\rho_2}^{\HH}\right)F(\rho)\left(I - P_{x,\rho_1,\rho_2}\right).
\]
It is easy to check  that
\[
	P_{x,\rho_1,\rho_2} x=x, \quad x^{\HH}\Phi(\rho_1,\rho_2) P_{x,\rho_1,\rho_2}=x^{\HH}\Phi(\rho_1,\rho_2),\quad  P_{x,\rho_1,\rho_2}^2=P_{x,\rho_1,\rho_2},
\]
i.e., $P_{x,\rho_1,\rho_2}$ is an (oblique) projection.

\section{Generic LOBPeCG Framework}
First we present a framework for LOBPeCG, namely Algorithm~\ref{alg:LOBPeCG}.
Note that in the shortcut $\LOCG(n_b,m_e)$,
$n_b$ represents the block size, i.e., the number the eigenpairs to compute simultaneously, while
$m_e$ indicates the size of the subspace extension so that $m_e+1$ is the dimension of the Krylov subspace.
\begin{algorithm}[ht]
	\caption{Locally optimal block preconditioned extended conjugate gradient method: $\LOCG(n_b,m_e)$}\label{alg:LOBPeCG}
	Given an initial proper approximation $X_0\in{\mathbb C}^{n\times n_b}$,
	and an integer $m_e\ge 1$, and a series of preconditioners $\{K_{i;j}\}$,
	the algorithm computes the approximations of the eigenpairs $(\lambda_j,u_j)$ for $j\in{\mathbb J}$, where ${\mathbb J}=\{1\le j\le n_b\}$  for computing the few smallest  eigenpairs.
	\vspace{2pt} \hrule
	\begin{algorithmic}[1]
		\State solve the projected problem for $X_0^{\HH}F(\lambda)X_0$ to get its  eigenpairs $(\rho_{0;j},y_j)$;
		\State $X_0 \,(\,=[\dots,x_{0;j},\dots])=X_0[y_1,\dots,y_{n_b}]$,  $X_{-1}=0$, $ {\mathbb J}=\{1\le j\le n_b\}$;
		\For{$i=0,1,\dots$}
		\State construct preconditioners $K_{i;j}$ for $j\in {\mathbb J}$;
		\State compute a basis matrix $Z_i$ of the subspace
			$\sum_{j\in {\mathbb J}}\krylov_{m_e}(K_{i;j}F(\rho_{i;j}),x_{i;j})+\range(X_{i-1})$;
		\State compute the $n_b$  proper eigenpairs of
		$Z_i^{\HH}F(\lambda)Z_i$:  $(\rho_{i+1;j},y_{i;j})$ for $j\in {\mathbb J}$
		and let $\Omega_{i+1}=\diag(\dots,\rho_{i+1;j},\dots)$ whose diagonal entries are those for $j\in {\mathbb J}$;
		\State $X_{i+1} \,(\,=[\dots,x_{i+1;j},\dots])=Z_iY_i$, where $Y_i=[\dots,y_{i;j},\dots]$ whose columns are those for $j\in {\mathbb J}$;
		\EndFor
		\State\Return approximate  eigenpairs to $(\lambda_j,u_j)$ for $j\in {\mathbb J}$.
	\end{algorithmic}
\end{algorithm}

We will deal with $\LOCG(1,m_e)$ in the following.
Since $j\equiv 1$, we can omit the index $j$ safely.

In every iteration of the algorithm, computing the proper eigenpairs of $Z^{\HH}F(\lambda)Z$ is equivalent to solving the following optimization problem:
\begin{equation} \label{eq:iteration-optimize}
	\rho_{i+1}= \rho(Z_iy_i) = \min_{
		\begin{subarray}{c}
			\text{proper $y$}
		\end{subarray}
	}\rho(Z_iy),
\end{equation}
where $Z_i$ is a basis of $\subspan\{x_{i},K_iF(\rho_{i})x_{i},\dots,(K_iF(\rho_{i}))^{m_e-1}x_{i},x_{i-1}\}$.

\begin{theorem} \label{thm:convergence:nb=1}
	Let the sequences $\{\rho_i\},\{x_i\},\{r_i:=F(\rho_i)x_i\}$ be produced by $\LOCG(1,m_e)$.
	Suppose that for all $i$, $y_i$ is a stationary point of $\rho(Z_iy)$.
	\begin{enumerate}
		\item Only one of the following two mutually exclusive situations can occur:
			\begin{enumerate}
				\item For some $i$, 
					$r_i=0$, and then
					$ \krylov_{m_e}(K_iF(\rho_{i}),x_{i})=\subspan\{x_{i}\}$ for $m_e\ge2$.
					Then we have
					\begin{equation}\label{eq:lucky-case}
						\rho_i=\rho_{i+1}=\dots,\,\,
						x_i=x_{i+1}=\dots,\,\,
						r_i=r_{i+1}=\dots=0,
					\end{equation}
					and $(\rho_i,x_i)$ is an eigenpair of $F(\lambda)$.
				\item $\rho_i$ is strictly monotonically decreasing, and
					$\rho_i\to\hat\rho\in[\lambda_-,\lambda_+]$ as $i\to\infty$, and
					$r_i\ne 0$ for all $i$, and no two $x_i$ are linearly dependent.
			\end{enumerate}
		\item $x_i^{\HH}r_i=0$, $Z_i^{\HH}r_{i+1}=0$.
		\item in the case of Item~1(b), if $\{x_i\}$ is bounded under the proper constraint, then
			\begin{enumerate}
				\item $r_i\ne 0$ for all $i$ but $r_i\to 0$ as $i\to\infty$,
				\item $\hat\rho$ is an eigenvalue of $F(\lambda)$, and any limit point $\hat x$ of $\{x_i\}$ is a corresponding
					eigenvector, i.e., $F(\hat\rho)\hat x=0$.
			\end{enumerate}
	\end{enumerate}
\end{theorem}
\begin{proof}
	The proof is nearly the same as its analogue by Liang and Li \cite[Theorem~8.1]{liangL2015hyperbolic}.
	First by \eqref{eq:iteration-optimize}, clearly $\rho_{i+1}\le\rho_i$.
	There are only two possibilities: either $r_i=0$ for some $i$ or $r_i\ne0$ for all $i$.
	If $r_i=F(\rho_{i})x_{i}=0$ for some $i$, then $\range(Z_i)=\subspan\{x_{i},x_{i-1}\}$.
	Note that $\range(Z_{i-1})=\subspan\{x_{i-1},K_{i-1}F(\rho_{i})x_{i-1},\dots,(K_{i-1}F(\rho_{i}))^{m_e-1}x_{i-1},x_{i-2}\}$
	and $x_i=Z_{i-1}y_{i-1}\in\range(Z_{i-1})$.
	Then $\range(Z_i)\subset\range(Z_{i-1})$, which implies $\rho_{i+1}=\rho_i$ and $x_{i+1}=x_i$ and then $r_{i+1}=r_i=0$.
	Thus, \eqref{eq:lucky-case} holds.
	Now consider $r_i\ne0$ for all $i$.
	Note that $r_i\ne 0$ implies $\nabla\rho_i\ne 0$, and so $\rho(x_i-\nu_1K_i\nabla\rho_i)<\rho(x_i)$ for some $\nu_1$ with sufficiently tiny $|\nu_1|$.
	This in turn implies $\rho(x_i+\nu_2r_i)<\rho(x_i)$ for some $\nu_2$ with sufficiently tiny $|\nu_2|$.
	Note that $x_i$ satisfies the proper constraint and the constraint is continuous, which implies $x_i+\nu_2r_i$ satisfies the proper constraint.
	Thus,
	\[
		\rho_{i+1}=\min_t\rho(Zy_i)\le\rho(x_i+\nu_2r_i)<\rho(x_i).
	\]
	Therefore $\rho_i$ is strictly monotonically decreasing.
	Since $\rho_i$ is strictly monotonically decreasing and
	bounded from below since $\rho_i\ge\lambda_-$, it is convergent and $\rho_i\to\hat\rho\in[\lambda_-,\lambda_+]$
	because $\rho_i=\rho(x_i)\in[\lambda_-,\lambda_+]$ for all $i$.
	No two $x_i$ are linear dependent because linear dependent $x_i$ and $x_j$ produce $\rho_i=\rho_j$. This proves Item~1.

	For Item~2,
	easy to see $x_i^{\HH}r_i=x_i^{\HH}F(\rho_i)x_i=0$.
	Since $y_i$ is a stationary point,
	\[
		Z_i^{\HH}r_{i+1}=-\frac{\sigma(x_{i+1})}{2}Z_i^{\HH}\nabla\rho(x_{i+1})=-\frac{\sigma(x_{i+1})}{2}Z_i^{\HH}\nabla\rho(Z_iy_i)=-\frac{\sigma(x_{i+1})}{2}\frac{\diff \rho(Z_iy_i)}{\diff y}=0.
	\]

	For Item~3(a), we have $\|r_i\|=\|F(\rho_i)x_i\|\le\left[\sum_{k=0}^m\|A_k\||\lambda_{\ell}|^{m-k}\right]\|x_i\|$
	so $\{r_i\}$ is a bounded sequence. It suffices to show that
	any limit point of $\{r_i\}$ is the zero vector. Assume, to the contrary, $\{r_i\}$
	has a nonzero limit point $\hat r$, i.e.,
	$r_{i_j}\to \hat r$, where $\{r_{i_j}\}$ is a subsequence of $\{r_i\}$. Since $\{x_{i_j}\}$ is bounded, it has a convergent subsequence. Without loss of generality, we may assume
	$x_{i_j}$ itself is convergent and $x_{i_j}\to\hat x$ as $j\to\infty$.
	We have
	$\hat r^{\HH}\hat x=0$ and $\hat x$ satisfies the proper constraint because
	$r_{i_j}^{\HH}x_{i_j}=0$ and $x_{i_j}$ satisfies the proper constraint.
	Now consider the projected problem for
	\begin{equation*}\label{eq:sequence:1d:basic:pf-1}
		F_{i_j}(\lambda):=Y_{i_j}^{\HH}F(\lambda)Y_{i_j}
		=\begin{bmatrix}
			x_{i_j}^{\HH}F(\lambda)x_{i_j} & x_{i_j}^{\HH}F(\lambda)r_{i_j} \\
			r_{i_j}^{\HH}F(\lambda)x_{i_j} & r_{i_j}^{\HH}F(\lambda)r_{i_j}
		\end{bmatrix},
	\end{equation*}
	where $Y_{i_j}=[x_{i_j},r_{i_j}]$.
	Since $r_{i_j}^{\HH}x_{i_j}=0$, $\rank(Y_{i_j})=2$, and thus $F_{i_j}(\lambda)$ still satisfies the assumptions at the beginning of the paper.
	Denote by $\mu_{j;k}$ its eigenvalues. It can be seen that
	\begin{equation}\label{eq:sequence:1d:basic:pf-2}
		\lambda_-<\lambda_1\le\mu_{j;1}\le\mu_{j;2}\le\lambda_{\ell}.
	\end{equation}
	Then 
	$\lambda_1\le\rho_{i_j+1}\le\mu_{j;1}$. Let
	\[
		\what F(\lambda)=\lim_{j\to\infty}F_{i_j}(\lambda)
	\]
	whose eigenvalues
	are denoted by $\hat\mu_i$. By the continuity of the eigenvalues with respect to the entries of coefficient matrices,
	we know $\mu_{j;i}\to\hat\mu_i$ as $j\to\infty$, and thus
	\begin{equation}\label{eq:sequence:1d:basic:pf-2a}
		\lambda_-<\lambda_1\le\hat\mu_1\le\hat\mu_2\le\hat\lambda_{\ell}.
	\end{equation}
	Notice by \eqref{eq:sequence:1d:basic:pf-2} and \eqref{eq:sequence:1d:basic:pf-2a}
	\begin{equation}\label{eq:sequence:1d:basic:pf-3}
		\lambda_1\le\rho_{i_j+1}\le\mu_{j;1}
		\quad\Rightarrow\quad
		\lambda_-<\lambda_1\le\hat\rho\le\hat\mu_1.
	\end{equation}
	On the other hand, by \eqref{eq:sequence:1d:basic:pf-2}, we have
	\[
		\what F(\hat\rho)=\lim_{j\to\infty}F_{i_j}(\rho_{i_j})
		=\lim_{j\to\infty}\begin{bmatrix}
			0			& r_{i_j}^{\HH}r_{i_j} \\
			r_{i_j}^{\HH}r_{i_j} & r_{i_j}^{\HH}F(\rho_{i_j})r_{i_j}
		\end{bmatrix}
		=\begin{bmatrix}
			0			& \hat r^{\HH}\hat r \\
			\hat r^{\HH}\hat r & \hat r^{\HH}F(\hat\rho)\hat r
		\end{bmatrix}
	\]
	which is indefinite because $\hat r^{\HH}\hat r>0$. But by \eqref{eq:sequence:1d:basic:pf-3},
	$\what F(\hat\rho)\preceq 0$, a contradiction. So $\hat r=0$, as was to be shown.

	For Item~3(b), since
	$\|x_i\|=1$, $\{x_i\}$ has at least one limit point.
	Let $\hat x$ be any limit point of $x_i$, i.e., $x_{i_j}\to \hat x$.
	Taking the limit on both sides of $F(\rho_{i_j})x_{i_j}=r_{i_j}$ yields
	$F(\hat\rho)\hat x=0$,
	i.e., $(\hat\rho,\hat x)$ is an eigenpair.
\end{proof}

Theorem~\ref{thm:convergence:nb=1} shows that $\LOCG(1,m_e)$ converges globally, but provides no information on its convergence rate.
In order to obtain such a rate, we proceed as follows:
first, a relationship between the quantities of two successive iterations is established in Theorem~\ref{thm:linesearch};
then, by this relationship, $\LOCG(1,m_e)$ is compared with $\SD(1,m_e)$ in Theorem~\ref{thm:rate:compare}, where $\SD(1,m_e)$ is the block preconditioned steepest descent method;
finally, the rate follows from this comparison in Theorem~\ref{thm:rate}.
These three theorems are reminiscent of the theorems by Ovtchinnikov
\cite[Theorem~2.6,Theorem~4.1, and Theorem~4.2]{ovtc:08:jacobi:I},
respectively.
Our theorems are more general than those 
w.r.t.\ three aspects:
they hold for any Hermitian matrix polynomial $F(\lambda)$ satisfying the assumptions in Section~\ref{sec:intro}, other than only the standard Hermitian eigenvalue problem $F(\lambda)=\lambda I - A$;
they allow for any $m_e$ in $\LOCG(1,m_e)$, other than only $m_e=1$;
the estimates are somewhat refined.
\begin{theorem} \label{thm:linesearch}
	Let $x\ne0,r(x)\ne0,p\ne0$, and $S=\begin{bmatrix} s^{(1)}&\dots&s^{(k)}
	\end{bmatrix}$, which satisfy $p^{\HH}r(x)\ne0, S^{\HH}r(x)=0$.
	Suppose that $ \begin{bmatrix} x & p & S
	\end{bmatrix} $ is of full column rank,
	and
	$(\alpha_{\opt},b_{\opt})$ is a stationary point of the function $\rho\left(x+\alpha (I-P_{x,\rho(x_{\opt}),\rho(x)})[p+Sb]\right)$.
	Write
	\[
		s = p + Sb_{\opt},\quad
		d = \alpha_{\opt} (I-P_{x,\rho(x_{\opt}),\rho(x)})s,\quad
		x_{\opt} = x + d.
	\]
	Then, for the nontrivial case that $x_{\opt}\ne x$,
	\begin{equation}\label{eq:thm:linesearch:r_opt-perp}
		\alpha_{\opt}\ne0,\quad r_{\opt}\perp\subspan\{x,p,S,s,d\},
	\end{equation}
	and
	\begin{align}
		\alpha_{\opt} &= -\frac{p^{\HH}r(x)}{s^{\HH}\check F_{\rho(x_{\opt}),\rho(x)}(\rho(x_{\opt});x)s} = -\frac{d^{\HH}F(\rho(x_{\opt}))d}{r(x)^{\HH}p}, \label{eq:thm:linesearch:alpha}
		\\ \rho(x_{\opt})-\rho(x) &= \frac{|r(x)^{\HH}p|^2}{\left[x^{\HH}\Phi (\rho(x_{\opt}),\rho(x))x\right]\left[s^{\HH}\check F_{\rho(x_{\opt}),\rho(x)}(\rho(x_{\opt});x)s\right]}
		= \frac{d^{\HH}F(\rho(x_{\opt}))d}{x^{\HH}\Phi (\rho(x_{\opt}),\rho(x))x}, 
		\label{eq:thm:linesearch:rho}
		\\ r(x_{\opt})-r(x) &= \check F_{\rho(x_{\opt}),\rho(x)}(\rho(x_{\opt});x)d, \label{eq:thm:linesearch:r}
		\\ b_{\opt} &= -\left[S^{\HH}\check F_{\rho(x_{\opt}),\rho(x)}(\rho(x_{\opt});x)S\right]^{\pinv}S^{\HH}\check F_{\rho(x_{\opt}),\rho(x)}(\rho(x_{\opt});x)p+v, \label{eq:thm:linesearch:beta}
	\end{align}
	where $v$ is a vector satisfying
	\begin{equation}\label{eq:thm:linesearch:v}
		\check F_{\rho(x_{\opt}),\rho(x)}(\rho(x_{\opt});x)Sv\perp\subspan\{x,p,S,s,d\},
	\end{equation}
	as long as
	\[
		x^{\HH}\Phi (\rho(x_{\opt}),\rho(x))x \ne0, \quad
		s^{\HH}\check F_{\rho(x_{\opt}),\rho(x)}(\rho(x_{\opt});x)s \ne0,
	\]
	Besides, \eqref{eq:thm:linesearch:alpha}--\eqref{eq:thm:linesearch:beta} also holds for the trivial case that $\alpha_{\opt}=0,d=0,x_{\opt}=x,\rho(x_{\opt})=\rho(x),r(x_{\opt})=r(x)$.
\end{theorem}
\begin{proof}
	Write
	\[
		\Phi_{\opt}=\Phi (\rho(x_{\opt}),\rho(x)),\quad P_{\opt}=P_{x,\rho(x_{\opt}),\rho(x)}=\frac{xx^{\HH}\Phi_{\opt}}{x^{\HH}\Phi_{\opt}x}.
	\]
	Recall from the end of Section~\ref{sec:intro} that we have
	\[
		 r(x_{\opt})^{\HH}x_{\opt}=0,\quad r(x)^{\HH}x=0,\quad r(x)^{\HH}P_{\opt}=0,\quad x^{\HH}\Phi_{\opt}P_{\opt}=x^{\HH}\Phi_{\opt}.
	\]
	Since $(\alpha_{\opt},b_{\opt})$ is a stationary point of the function $\rho$,
	\begin{equation} \label{eq:thm:linesearch:prf:diff-beta}
		\begin{aligned}[b]
			0 = \frac{\diff}{\diff b}\rho\left(x+\alpha_{\opt} (I-P_{\opt})(p+ Sb_{\opt})\right)
			&= \left(\nabla \rho(x+\alpha_{\opt} (I-P_{\opt})s)\right)^{\HH}\alpha_{\opt}(I-P_{\opt})S
			\\&= - \frac{2\alpha_{\opt}}{\sigma(x_{\opt})}r(x_{\opt})^{\HH}(I-P_{\opt})S
			,
		\end{aligned}
	\end{equation}
	and
	\begin{equation} \label{eq:thm:linesearch:prf:diff-alpha}
		\begin{aligned}[b]
			0 = \frac{\diff}{\diff \alpha}\rho(x+\alpha_{\opt} (I-P_{\opt})s)
			&= \left(\nabla \rho(x+\alpha_{\opt} (I-P_{\opt})s)\right)^{\HH}(I-P_{\opt})s
			\\&= - \frac{2}{\sigma(x_{\opt})}r(x_{\opt})^{\HH}(I-P_{\opt})s
			,
		\end{aligned}
	\end{equation}
	which means $r(x_{\opt})^{\HH}d=0$.
	Then $r(x_{\opt})^{\HH}x=r(x_{\opt})^{\HH}(x_{\opt}-d)=0$ and $r(x_{\opt})^{\HH}P_{\opt}=0$.
	Thus, $r(x_{\opt})^{\HH}s=0$ by \eqref{eq:thm:linesearch:prf:diff-alpha} and $r(x_{\opt})^{\HH}S=0$ by \eqref{eq:thm:linesearch:prf:diff-beta}, so that $r(x_{\opt})^{\HH}p=r(x_{\opt})^{\HH}(s-Sb_{\opt})=0$. Then \eqref{eq:thm:linesearch:r_opt-perp} holds.
	According to \eqref{eq:thm:linesearch:prf:diff-alpha},
	\begin{equation} \label{eq:thm:linesearch:prf:diff-alpha:2}
		\begin{aligned}[b]
			0
			&= x_{\opt}^{\HH}F(\rho(x_{\opt}))(I-P_{\opt})s
			\\&= x^{\HH}F(\rho(x_{\opt}))(I-P_{\opt})s+\overline{\alpha_{\opt}}s^{\HH}\check F_{\rho(x_{\opt}),\rho(x)}(\rho(x_{\opt});x)s.
		\end{aligned}
	\end{equation}
	Note that
	$ d^{\HH} F(\rho(x_{\opt}))d =
	|\alpha_{\opt}|^2s^{\HH}\check F_{\rho(x_{\opt}),\rho(x)}(\rho(x_{\opt});x)s
	$
	and
	\begin{equation} \label{eq:thm:linesearch:prf:rs}
		\begin{aligned}[b]
			x^{\HH}F(\rho(x_{\opt}))(I-P_{\opt})s
			&=  x^{\HH}\left[F(\rho(x_{\opt}))-F(\rho(x))\right](I-P_{\opt})s + x^{\HH}F(\rho(x))(I-P_{\opt})s
			\\&=  [\rho(x_{\opt})-\rho(x)]x^{\HH}\Phi_{\opt}(I-P_{\opt})s + r(x)^{\HH}(I-P_{\opt})s
			\\&=  r(x)^{\HH}s
			=  r(x)^{\HH}p
			.
		\end{aligned}
	\end{equation}
	Then by \eqref{eq:thm:linesearch:prf:diff-alpha:2}, we have \eqref{eq:thm:linesearch:alpha}.
	Since
	\begin{equation} \label{eq:thm:linesearch:prf:diff-r}
		\begin{aligned}[b]
			r(x_{\opt})-r(x)
			&= F(\rho(x_{\opt}))x_{\opt}-F(\rho(x))x
			\\&= F(\rho(x_{\opt}))x_{\opt}-F(\rho(x_{\opt}))x+F(\rho(x_{\opt}))x-F(\rho(x))x
			\\&= \alpha_{\opt}F(\rho(x_{\opt}))(I-P_{\opt})s + [\rho(x_{\opt})-\rho(x)]\Phi_{\opt}x
			,
		\end{aligned}
	\end{equation}
	by \eqref{eq:thm:linesearch:prf:rs}, we get
	\begin{equation} \label{eq:thm:linesearch:prf:diff-rho}
			\rho(x_{\opt})-\rho(x)
			=\frac{x^{\HH}(r(x_{\opt})-r(x)) - \alpha_{\opt}x^{\HH}F(\rho(x_{\opt}))(I-P_{\opt})s}{ x^{\HH}\Phi_{\opt}x}
			= - \frac{\alpha_{\opt}r(x)^{\HH}p}{ x^{\HH}\Phi_{\opt}x}
			.
	\end{equation}
	Together with \eqref{eq:thm:linesearch:alpha}, we obtain \eqref{eq:thm:linesearch:rho}. 
	Further,
	\begin{align*}
		\check F_{\rho(x_{\opt}),\rho(x)}(\rho(x_{\opt});x)d
		&= \alpha_{\opt}(I-P_{\opt}^{\HH})F(\rho(x_{\opt}))(I-P_{\opt})s
		\\&= \alpha_{\opt}\left(F(\rho(x_{\opt}))(I-P_{\opt})s - \frac{\Phi_{\opt}xx^{\HH}}{x^{\HH}\Phi_{\opt}x}F(\rho(x_{\opt}))(I-P_{\opt})s\right)
		\\&= \alpha_{\opt}\left(F(\rho(x_{\opt}))(I-P_{\opt})s - \Phi_{\opt}x\frac{r(x)^{\HH}p}{x^{\HH}\Phi_{\opt}x}\right)\qquad \text{by \eqref{eq:thm:linesearch:prf:rs}}
		\\&= \alpha_{\opt}\left(F(\rho(x_{\opt}))(I-P_{\opt})s + \Phi_{\opt}x\frac{\rho(x_{\opt})-\rho(x)}{\alpha_{\opt}}\right)\qquad \text{by \eqref{eq:thm:linesearch:prf:diff-rho}}
		\\&= r(x_{\opt})-r(x),\qquad \text{by \eqref{eq:thm:linesearch:prf:diff-r}}
	\end{align*}
	hence we obtain \eqref{eq:thm:linesearch:r}.
	Thus,
	\begin{align*}
		(x_{\opt}-x)^{\HH}\check F_{\rho(x_{\opt}),\rho(x)}(\rho(x_{\opt});x)(x_{\opt}-x)
		&= (r(x_{\opt})-r(x))^{\HH}(x_{\opt}-x)
		\\&= -r(x)^{\HH}(x_{\opt}-x)
		= -\alpha_{\opt}r(x)^{\HH}(I-P_{\opt})s
		= -\alpha_{\opt}r(x)^{\HH}p
		.
	\end{align*}
	Finally,
	\begin{align*}
		0 &= S^{\HH}(I-P_{\opt}^{\HH})(r(x_{\opt})-r(x)) \qquad\text{by \eqref{eq:thm:linesearch:prf:diff-beta}}
		\\&= S^{\HH}(I-P_{\opt}^{\HH})\check F_{\rho(x_{\opt}),\rho(x)}(\rho(x_{\opt});x)\alpha_{\opt}(I-P_{\opt})(p+Sb) \qquad\text{by \eqref{eq:thm:linesearch:prf:diff-r}}
		\\&= \alpha_{\opt} S^{\HH}\check F_{\rho(x_{\opt}),\rho(x)}(\rho(x_{\opt});x)(p+Sb)
		,\qquad\text{by the definition of $\check F$}
	\end{align*}
	which implies \eqref{eq:thm:linesearch:beta},
	and $S^{\HH}\check F_{\rho(x_{\opt}),\rho(x)}(\rho(x_{\opt});x)Sv=0$.
	Note that $S^{\HH}\check F_{\rho(x_{\opt}),\rho(x)}d=S^{\HH}[r(x_{\opt})-r(x)]=0$ by \eqref{eq:thm:linesearch:r},
	and $\check F_{\rho(x_{\opt}),\rho(x)}x=\check F(\rho(x_{\opt}))(I-P_{\opt})x=0$.
	It is easy to obtain \eqref{eq:thm:linesearch:v}.
\end{proof}
\begin{theorem} \label{thm:rate:compare}
	Suppose $\lambda_1<\rho_i<\lambda_2$.
	Assume that $K_i^{1/2}F'(\rho_i)K_i^{1/2}$ is positive definite in the search subspace, or equivalently, $Z_i^{\HH}K_i^{1/2}F'(\rho_i)K_i^{1/2}Z_i\succ0$.
	If $\rho_{i-1}-\lambda_1$ is sufficiently small,
	then for $\LOCG(1,1)$,
	either $\rho_i-\rho_{i+1}\ge\sqrt{\rho_{i-1}-\rho_i}$,
	or
	\begin{equation}\label{eq:thm:rate:compare}
		\frac{1}{\rho_i-\rho_{i+1}}+ \frac{1}{\rho_{i-1}-\rho_i}=
		\frac{1+O(\sqrt{\rho_i-\rho_{i+1}})+O(\rho_{i-1}-\rho_i)}{\rho_i-\rho_{i+1}^{\eSD}},
	\end{equation}
	where $\rho_{i+1}^{\eSD}$ is the minimal value of $\rho(x)$ in the subspace $\krylov_{m_e}(K_iF(\rho_i),x_i)$.
\end{theorem}
\begin{remark}\label{rk:thm:rate:compare}
	If the case that $\rho_i-\rho_{i+1}\ge\sqrt{\rho_{i-1}-\rho_i}$ occurs,  the $i$th iteration improves the approximation a lot,
	so it is very exceptional.
\end{remark}
\begin{proof}
	Assume that $\rho_i-\rho_{i+1}\ge\sqrt{\rho_{i-1}-\rho_i}$ fails, namely
	\begin{equation}\label{eq:thm:rate:compare:prf:assumption}
		\rho_i-\rho_{i+1}<\sqrt{\rho_{i-1}-\rho_i}.
	\end{equation}
	For a general $K_i\succ0$, the $i$th iteration is just equivalent to
	the $i$th iteration of the algorithm applied to $K_i^{1/2}F(\lambda)K_i^{1/2}$ without a preconditioner,
	and then everything below can be easily examined.
	Thus, in the following we assume $K_i=I$.

	To use Theorem~\ref{thm:linesearch}, without loss of generality,
	suppose we normalize $x_i$ in every iteration to make the first element of $y_i$ (in Step~6 of Algorithm~\ref{alg:LOBPeCG}) be $1$.
	Then in the $i$th iteration,
	write
	\begin{gather*}
		\varepsilon_i=\rho_i-\lambda_1,\quad
		\delta_i=-(\rho_{i+1}-\rho_i)\ge0,\quad
		d_i = x_{i+1} - x_i,\quad
		F_i = F(\rho_i),\quad
		F'_i=F'(\rho_i),
		\\
		\Phi_i=\Phi(\rho_{i+1},\rho_i),\quad
		P_i=P_{x_i,\rho_{i+1},\rho_i},\quad
		\check F_i=\check F_{\rho_{i+1},\rho_i}(\rho_{i+1};x_i).
	\end{gather*}
	Clearly $d_{i-1}^{\HH}r_i=0$.
	Note that $\sigma(x_i)=x_i^{\HH}F'_ix_i>0$.
	Thus,
	\[
		\frac{x_i^{\HH}\Phi_ix_i}{x_i^{\HH}x_i}=\frac{\sigma(x_i)}{x_i^{\HH}x_i}+\sum_{k=2}^m\frac{(-\delta_i)^{k-1}}{k!}\frac{x_i^{\HH}F^{(k)}(\rho_i)x_i}{x_i^{\HH}x_i}=\frac{\sigma(x_i)}{x_i^{\HH}x_i}+O(\delta_i)>0.
	\]
	Without loss of generality, we assume $x_i^{\HH}\Phi_ix_i=1$.

	If the notations in Theorem~\ref{thm:linesearch} are adopted, then
	\[
		x_i = x,\quad  r_i = r(x),\quad   \rho_i = \rho(x), \qquad
		 x_{i+1} = x_{\opt},\quad  r_{i+1} = r(x_{\opt}),\quad   \rho_{i+1} = \rho(x_{i+1}). 
	 \]
	For
	\[
		S_i := (I-P_i)\left(I-\frac{r_ir_i^{\HH}}{r_i^{\HH}r_i}\right)\begin{bmatrix} F_ir_i & \dots & F_i^{m_e}r_i
		\end{bmatrix},
	\]
	we obtain $r_i^{\HH}S_i=0$.
	$\rho_{i+1}$ can be recognized as $\rho_{\opt}$ in Theorem~\ref{thm:linesearch} as we let
	\[
		p = r_i,\quad
		S = \begin{bmatrix} x_{i-1} & S_i
		\end{bmatrix}=:\wtd S_i.
	\]
	Without loss of generality, assume $\begin{bmatrix}
		r_i & x_{i-1} & S_i
	\end{bmatrix}$ is of full column rank, otherwise we can delete the last several columns of $S_i$, which will not affect the search process.
	Thus, by \eqref{eq:thm:linesearch:rho} and \eqref{eq:thm:linesearch:beta},
	\[
		\delta_i =\rho_i-\rho_{i+1} = -\frac{|r_i^{\HH}r_i|^2}{[x_i^{\HH}\Phi_ix_i][s_i^{\HH}\check F_is_i]},
	\]
	where
	\[
		s_i = r_i - \wtd S_i(\wtd S_i^{\HH}\check F_i\wtd S_i)^{\pinv}\wtd S_i^{\HH}\check F_ir_i+\wtd S_i v_i,
		\qquad \check F_i\wtd S_iv_i\perp\subspan\{x_i,r_i,\wtd S_i,s_i,d_i\}.
	\]

	To describe the search process in the subspace $\krylov_{m_e}(K_iF(\rho_i),x_i)$, we use the superscript ``$\cdot^{\eSD}$'' for certain terms, which gives
	\begin{gather*}
		\delta_i^{\eSD}=-(\rho_{i+1}^{\eSD}-\rho_i)\ge0,\quad
		F_{i+1}^{\eSD} = F(\rho_{i+1}^{\eSD}),\quad
		\Phi_i^{\eSD}=\Phi(\rho_{i+1}^{\eSD},\rho_i),\quad
		P_i^{\eSD}=P_{x_i,\rho_{i+1}^{\eSD},\rho_i},\quad
		\check F_i^{\eSD}=\check F_{\rho_{i+1}^{\eSD},\rho_i}(\rho_{i+1}^{\eSD};x_i).
		\\
		x_{i+1}^{\eSD} = x_{\opt}^{\eSD},\quad  r_{i+1}^{\eSD} = r(x_{\opt}^{\eSD}),\quad  \rho_{i+1}^{\eSD} = \rho(x_{i+1}^{\eSD}), 
	\end{gather*}
	Similarly, $\rho_{i+1}^{\eSD}$ can be recognized as $\rho_{\opt}^{\eSD}$ in Theorem~\ref{thm:linesearch} as we let
	\[
		p^{\eSD} = r_i,\quad
		S^{\eSD} = S_i.
	\]
	Thus, by \eqref{eq:thm:linesearch:rho} and \eqref{eq:thm:linesearch:beta},
	\begin{equation}\label{eq:thm:rate:compare:prf:delta:eSD}
		\delta_i^{\eSD}=\rho_i-\rho_{i+1}^{\eSD} = -\frac{|r_i^{\HH}r_i|^2}{[x_i^{\HH}\Phi_i^{\eSD}x_i][(s_i^{\eSD})^{\HH}\check F_i^{\eSD}s_i^{\eSD}]}.
	\end{equation}
	where
	\[
		s_i^{\eSD} = r_i - S_i(S_i^{\HH}\check F_i^{\eSD}S_i)^{\pinv}S_i^{\HH}\check F_i^{\eSD}r_i+S_iv_i^{\eSD}, 
		\qquad \check F_i^{\eSD} S_iv_i^{\eSD}\perp\subspan\{x_i,r_i,S_i,s_i^{\eSD}\}.
	\]
	The rest of the proof is to estimate the ratio of $\delta_i^{\eSD}$ and $\delta_i$.
	Let
	\[
		\kappa:=\frac{\delta_i^{\eSD}}{\delta_i}
		=
		\frac{x_i^{\HH}\Phi_ix_i}{x_i^{\HH}\Phi_i^{\eSD}x_i}
		\frac{s_i^{\HH}\check F_is_i}{(s_i^{\eSD})^{\HH}\check F_i^{\eSD}s_i^{\eSD}}.
	\]
	Clearly, $\kappa\le1$.

	First, we prove that
	\begin{equation}\label{eq:thm:rate:compare:prf:claim:inv}
		\text{$S_i^{\HH}F_{i+1}S_i$ and $S_i^{\HH}F_{i+1}^{\eSD}S_i$ are nonsingular.}
	\end{equation}
	Write
	\begin{alignat*}{2}
		T_i &= S_i(S_i^{\HH}\check F_iS_i)^{-1}S_i^{\HH}\check F_i &&= S_i(S_i^{\HH}F_{i+1}S_i)^{-1}S_i^{\HH}F_{i+1}(I-P_i),
		\\
		T_i^{\eSD} &= S_i(S_i^{\HH}\check F_i^{\eSD}S_i)^{-1}S_i^{\HH}\check F_i^{\eSD} &&= S_i(S_i^{\HH}F_{i+1}^{\eSD}S_i)^{-1}S_i^{\HH}F_{i+1}^{\eSD}(I-P_i^{\eSD}).
	\end{alignat*}
	Clearly, $P_iS_i=0$, $P_iT_i=T_iP_i=0$, $P_i^{\eSD}T_i^{\eSD}=T_i^{\eSD}P_i^{\eSD}=0$, and
	\[
		T_i^{\HH}\check F_i=T_i^{\HH}\check F_iT_i,\quad
	(I-T_i^{\HH})\check F_i(I-T_i)=\check F_i(I-T_i)=(I-T_i^{\HH})\check F_i.
	\]
	We have $v_i^{\eSD}=0$ and
	\[
		s_i^{\eSD} = r_i - S_i(S_i^{\HH}\check F_i^{\eSD}S_i)^{-1}S_i^{\HH}\check F_i^{\eSD}r_i=(I-T_i^{\eSD})r_i.
	\]
	On the other hand, it is easy to see that
	$\wtd S_i^{\HH}\check F_i\wtd S_i$ is nonsingular if and only if
	\begin{equation}\label{eq:thm:rate:compare:prf:claim:yFx}
		\tau_i :=x_{i-1}^{\HH}\check F_i(I-T_i)x_{i-1}\ne0,
	\end{equation}
	and when it is nonsingular, that
	\[
		(\wtd S_i^{\HH}\check F_i\wtd S_i)^{-1}=
		\begin{bmatrix}
			x_{i-1}^{\HH}\check F_ix_{i-1} & x_{i-1}^{\HH}\check F_iS_i \\
			S_i^{\HH}\check F_ix_{i-1} & S_i^{\HH}\check F_iS_i\\
		\end{bmatrix}^{-1} =
		\begin{bmatrix}
			\frac{1}{\tau_i } & -\frac{1}{\tau_i } w_i ^{\HH} \\
			-\frac{1}{\tau_i } w_i  & \frac{1}{\tau_i } w_i w_i ^{\HH}+(S_i^{\HH}\check F_iS_i)^{-1}
		\end{bmatrix}
		,
	\]
	where $w_i = (S_i^{\HH}\check F_iS_i)^{-1}S_i^{\HH}\check F_ix_{i-1}$ satisfying $S_iw_i =T_ix_{i-1}$. 
	Actually, \eqref{eq:thm:rate:compare:prf:claim:yFx} is guaranteed by the claim \eqref{eq:thm:rate:compare:prf:claim:kappa4:dd} below.
	Thus, $\wtd S_i^{\HH}\check F_i\wtd S_i$ is nonsingular, 
	\[
		\wtd S(\wtd S^{\HH}\check F_i\wtd S)^{-1}\wtd S^{\HH}
		= S_i(S_i^{\HH}\check F_iS_i)^{-1}S_i^{\HH}+\frac{1}{\tau_i } (I-T_i)x_{i-1}x_{i-1}^{\HH}(I-T^{\HH})
		,
	\]
	and
	\begin{align*}
		s_i &= r_i - \wtd S_i(\wtd S_i^{\HH}\check F_i\wtd S_i)^{-1}\wtd S_i^{\HH}\check F_ir_i
		\\ &= r_i - S_i(S_i^{\HH}\check F_iS_i)^{-1}S_i^{\HH}\check F_ir_i-\frac{1}{\tau_i} (I-T_i)x_{i-1}x_{i-1}^{\HH}(I-T_i^{\HH})\check F_ir_i
		\\& = (I-T_i)\left[r_i-\frac{x_{i-1}^{\HH}\check F_i(I-T_i)r_i}{\tau_i}x_{i-1}\right].
	\end{align*}
	Write
	\[
		e_i= T_i^{\eSD}r_i-T_ir_i,
		\quad
		\beta_i = x_{i-1}^{\HH}\check F_i(I-T_i)r_i,
	\]
	so that $s_i^{\eSD}+e_i=(I-T_i)r_i$ and
	\[
		s_i = s_i^{\eSD}+e_i -\frac{\beta_i}{\tau_i}(I-T_i)x_{i-1}.
	\]
	Let
	\begin{equation} \label{eq:thm:rate:compare:prf:compare}
		\kappa=\frac{\delta_i^{\eSD}}{\delta_i}
		=
		\frac{x_i^{\HH}\Phi_ix_i}{x_i^{\HH}\Phi_i^{\eSD}x_i}
		\frac{(s_i^{\eSD})^{\HH}\check F_is_i^{\eSD}}{(s_i^{\eSD})^{\HH}\check F_i^{\eSD}s_i^{\eSD}}
		\frac{(s_i^{\eSD}+e_i)^{\HH}\check F_i(s_i^{\eSD}+e_i)}{(s_i^{\eSD})^{\HH}\check F_is_i^{\eSD}}
		\frac{s_i^{\HH}\check F_is_i}{(s_i^{\eSD}+e_i)^{\HH}\check F_i(s_i^{\eSD}+e_i)}
		=: \kappa_1\kappa_2\kappa_3\kappa_4.
	\end{equation}
	First, observe that
	\[
		\kappa_1
		= \frac{1}{x_i^{\HH}\Phi_i^{\eSD}x_i}
		= \frac{1}{\sigma(x_i)+O(\delta_i^{\eSD})}
		= \frac{1}{1+O(\delta_i)}
		=1+O(\delta_i).
	\]
	We assume for now that
	\begin{equation}\label{eq:thm:rate:compare:prf:claim:kappa23}
		\kappa_2=1+O(\delta_i),\quad
		\kappa_3=1+O(\delta_i).
	\end{equation}
	For $\kappa_4$,
	since $(I-T_i^{\HH})\check F_i(I-T_i)=\check F_i(I-T_i)=(I-T_i^{\HH})\check F_i$, we get
	then
	\[
		\kappa_4
		= \frac{s_i^{\HH}\check F_is_i}{(s_i^{\eSD}+e_i)^{\HH}\check F_i(s_i^{\eSD}+e_i)}
		= \frac{r_i^{\HH}\check F_i(I-T_i)r_i-\frac{\beta_i^2}{\tau_i}}{r_i^{\HH}\check F_i(I-T_i)r_i}
		= 1-\frac{\beta_i^2}{\tau_ir_i^{\HH}\check F_i(I-T_i)r_i}.
	\]
	We claim that
	\begin{align}
		\tau_i&= -\left[1+O(\delta_{i-1}^{1/2})+O(\delta_i)\right]\delta_{i-1},
		\label{eq:thm:rate:compare:prf:claim:kappa4:dd}
		\\
		\beta_i&= \left[1+O(\delta_{i-1})+O(\delta_i\delta_{i-1}^{1/2})\right]\|r_i\|^2+\left[ O(\delta_{i-1})+O(\delta_i\delta_{i-1}^{1/2}) \right]\|r_i\|,
		\label{eq:thm:rate:compare:prf:claim:kappa4:dr}
	\end{align}
	and
	\begin{equation}\label{eq:thm:rate:compare:prf:claim:kappa4:rFr/rr}
		-r_i^{\HH}\check F_i(I-T_i)r_i\sim r_i^{\HH}r_i = O(\delta_i).
	\end{equation}
	Recall \eqref{eq:thm:rate:compare:prf:assumption}, namely $\delta_{i-1}>\delta_i^2$.
	Therefore,
	\[
		\tau_i= -\left[1+O(\delta_{i-1}^{1/2})\right]\delta_{i-1},
		\qquad
		\beta_i= \left[1+O(\delta_{i-1})\right]\|r_i\|^2+ O(\delta_{i-1}) \|r_i\|.
	\]
	Thus,
	\begin{align*}
		1-\kappa_4
		= \frac{\beta_i^2}{\tau_ir_i^{\HH}\check F_i(I-T_i)r_i}
		&= \frac{\left(O(\delta_{i-1})\|r_i\|+\left[1+O(\delta_{i-1})\right]\|r_i\|^2\right)^2}{-\delta_{i-1}\left[1+O(\delta_{i-1}^{1/2})\right]r_i^{\HH}\check F_i(I-T_i)r_i}
		\\&= \frac{\left[1+O(\delta_{i-1})\right]\|r_i\|^4+O(\delta_{i-1})\|r_i\|^3+O(\delta_{i-1}^2)\|r_i\|^2}{-\delta_{i-1}\left[1+O(\delta_{i-1}^{1/2})\right]r_i^{\HH}\check F_i(I-T_i)r_i}
		\\&= \frac{\left[1+O(\delta_{i-1}^{1/2})\right]\|r_i\|^4+O(\delta_{i-1})\|r_i\|^3+O(\delta_{i-1}^2)\|r_i\|^2}{-\delta_{i-1}r_i^{\HH}\check F_i(I-T_i)r_i}
		\\&= \frac{1+O(\delta_{i-1}^{1/2})}{\delta_{i-1}}\frac{\|r_i\|^4}{r_i^{\HH}\check F_i(I-T_i)r_i}+O(\delta_i^{1/2})+O(\delta_{i-1})
		.
	\end{align*}
	By \eqref{eq:thm:rate:compare:prf:delta:eSD} and \eqref{eq:thm:rate:compare:prf:compare},
	\[
		\frac{\delta_i^{\eSD}}{\kappa_1\kappa_2\kappa_3}=\frac{\|r_i\|^4}{r_i^{\HH}\check F_i(I-T_i)r_i},
	\]
	which implies
	\[
		1-\kappa_4
		= \frac{1+O(\delta_{i-1}^{1/2})}{\delta_{i-1}}\frac{\delta_i^{\eSD}}{\kappa_1\kappa_2\kappa_3}+O(\delta_i^{1/2})+O(\delta_{i-1})
		.
	\]
	Since $(1-\kappa_4)\kappa_1\kappa_2\kappa_3=\kappa_1\kappa_2\kappa_3-\kappa=1-\kappa+O(\delta_i)$,
	we obtain
	\[
		\frac{\delta_i^{\eSD}}{\delta_{i-1}}+O(\delta_i^{1/2})\sqrt{\frac{\delta_i^{\eSD}}{\delta_{i-1}}}+O(\delta_i^{1/2})+O(\delta_{i-1})+O(\delta_i)-(1-\kappa)=0
		,
	\]
	which implies
	\begin{align*}
		\frac{\delta_i^{\eSD}}{\delta_{i-1}}
		&= \left(\frac{1}{2}\left[-O(\delta_i^{1/2})\pm \sqrt{O(\delta_i)+O(\delta_i^{1/2})+O(\delta_{i-1})+4(1-\kappa)}\right]\right)^2
		\\&= O(\delta_i)+O(\delta_i^{1/2})+O(\delta_{i-1})+4(1-\kappa)+2 O(\delta_i^{1/2})\sqrt{O(\delta_i^{1/2})+O(\delta_{i-1})+(1-\kappa)}
		\\&= 1-\kappa+O(\delta_i^{1/2})+O(\delta_{i-1}).
	\end{align*}
	With
	\[
		\frac{1}{\delta_i}
		=\frac{1+O(\delta_i^{1/2})+O(\delta_{i-1})}{\delta_i^{\eSD}}- \frac{1}{\delta_{i-1}}
		,
	\]
	we arrive at \eqref{eq:thm:rate:compare}.

	We defer the proofs of the claims \eqref{eq:thm:rate:compare:prf:claim:inv}, \eqref{eq:thm:rate:compare:prf:claim:kappa23}, \eqref{eq:thm:rate:compare:prf:claim:kappa4:dd}, \eqref{eq:thm:rate:compare:prf:claim:kappa4:dr}, and \eqref{eq:thm:rate:compare:prf:claim:kappa4:rFr/rr} to Appendix~\ref{sec:claim-proof}, as these consist of rather technical calculations and estimations.
\end{proof}

We summarize the findings of this section in the following theorem.
\begin{theorem}\label{thm:rate}
	Suppose $\lambda_1\le\rho_0<\lambda_2$.
	Let $\{\rho_i\}$ and $\{\rho_i^{\eSD}\}$ be produced by $\LOCG(1,m_e)$ and $\SD(1,m_e)$ with a fixed preconditioner $K\succ0$, respectively.
	Assume that $Z_i^{\HH}K^{1/2}F'(\lambda_1)K^{1/2}Z_i\succ0$.
	If $\rho_{i-1}-\lambda_1$ is sufficiently small,
	provided
	\[
		\rho_{i+1}^{\eSD}-\lambda_1\le \eta_{\eSD}(\rho_i^{\eSD}-\lambda_1)+O((\rho_i^{\eSD}-\lambda_1)^{3/2}),\quad\text{for all $i$ and a given $\eta_{\eSD}<1$},
	\]
	then
	\begin{equation}\label{eq:thm:rate}
		\rho_{i+1}-\lambda_1\le \eta^2(\rho_{i-1}-\lambda_1)+O((\rho_{i-1}-\lambda_1)^{3/2}),
	\end{equation}
	where
	\[
		\eta=\frac{\eta_{\eSD}}{2-\eta_{\eSD}}.
	\]
\end{theorem}
\begin{proof}
	The proof is exactly the same as its analogue by Ovtchinnikov \cite[Theorem~4.2]{ovtc:08:jacobi:I}.
\end{proof}

\section{Application to Definite Pairs and Hyperbolic Quadratic Polynomials}
\subsection{Definite Matrix Pair}
As we stated in Section~\ref{sec:intro}, the definite pair $F(\lambda)=\lambda B-A$ for the special case that $F(\lambda_0)\prec0$, $\interval=(\lambda_0,+\infty)$, and the smallest positive-type eigenvalue is chosen here. This setting satisfies the assumptions needed to apply tthe results from the previous section.
However, with little effort, we see that any definite pair or any type of eigenvalues could be transformed into the case mentioned before.
For example, for $F(\lambda_0)\prec0$, $\interval=(-\infty,\lambda_0)$, we consider $\what F(\lambda)=F(-\lambda)$ and $\what{\interval}=(-\lambda_0,+\infty)$;
for $F(\lambda_0)\succ0$, $\interval=(\lambda_0,+\infty)$, we consider $\what F(\lambda)=-F(\lambda)$ and $\what{\interval}=\interval$.

\begin{theorem}\label{thm:defpair}
	Let $\{\rho_i\},\{x_i\}$ be produced by $\LOCG(1,m_e)$ with a fixed preconditioner $K\succ0$ for the definite matrix pair $F(\lambda)=\lambda B - A$.
	Suppose $\lambda_1^+\le\rho_0<\lambda_2^+$.
	Assume that $Z_i^{\HH}K^{1/2}F'(\lambda_1)K^{1/2}Z_i\succ0$.
	\begin{enumerate}
		\item As $i\to\infty$, $\rho_i$ monotonically converges to $\lambda_1^+$, and $x_i$ converges to the corresponding eigenvector in direction, i.e., $F(\rho_i)x_i\to0$.
		\item Denote by $\gamma$ and $\Gamma$ the smallest and largest positive eigenvalue of the matrix $-KF(\lambda_1)$.
			If $\rho_i-\lambda_1^+$ is sufficiently small, then
			\begin{equation} \label{eq:thm:defpair-HQEP}
				\rho_{i+1} -\lambda_1^+\le\eta^2(\rho_{i-1}-\lambda_1^+)+O((\rho_{i-1}-\lambda_1^+)^{3/2}),
			\end{equation}
			where
			\[
				\eta=\frac{2}{\Delta^{2m_e}+\Delta^{-2m_e}}, \quad
				\Delta=\frac{\sqrt{\kappa}+1}{\sqrt{\kappa}-1}, \quad
				\kappa=\frac{\Gamma}{\gamma}.
			\]
	\end{enumerate}
\end{theorem}
\begin{proof}
	For a definite matrix pair,
	the optimization problem \eqref{eq:iteration-optimize} is
	\begin{equation*} \label{eq:iteration-optimize:defpair}
		\rho_{i+1}= \rho(Z_iy_i) = \min_{
			\begin{subarray}{c}
				\text{$y^{\HH}Z_i^{\HH}BZ_iy=1$}
			\end{subarray}
		}y^{\HH}Z_i^{\HH}AZ_iy.
	\end{equation*}
	Using Lagrangian multipliers, it is equivalent to
	\begin{equation*} \label{eq:iteration-optimize:lagrange:defpair}
		\rho_{i+1}= \min {\cal L}(y,\mu)=\min
		y^{\HH}Z_i^{\HH}AZ_iy-\mu(y^{\HH}Z_i^{\HH}BZ_iy-1).
	\end{equation*}
	The minimal point $(y_i,\mu_i)$ must satisfy:
	\begin{subequations}\label{eq:iteration-optimize:langrange:ymu:defpair}
		\begin{align}
			\frac{\partial {\cal L}(y_i,\mu_i)}{\partial y} &= 2Z_i^{\HH}AZ_iy_i-2\mu_i Z_i^{\HH}BZ_iy_i=0,\label{eq:iteration-optimize:langrange:y:defpair}\\
			\frac{\partial {\cal L}(y_i,\mu_i)}{\partial \mu} &= y_i^{\HH}Z_i^{\HH}BZ_iy_i-1=0.\label{eq:iteration-optimize:langrange:mu:defpair}
		\end{align}
	\end{subequations}
	Left multiplying \eqref{eq:iteration-optimize:langrange:y:defpair} by $y_i^{\HH}$ gives $\mu_i=\rho(Z_iy_i)$,
	and then $Z_i^{\HH}r(Z_iy_i)=Z_i^{\HH}F(\rho(Z_iy_i))Z_iy_i=0$.
	Thus, $\frac{\diff \rho(Z_iy_i)}{\diff y}=Z_i^{\HH}\nabla\rho(Z_iy_i)=0$, which means $y_i$ is a stationary point of $\rho(Z_iy)$.
	Besides, under the constraint $x_i^{\HH}Bx_i=1$,
	$$x_i^{\HH}(\lambda_-B-A)x_i=(\lambda_--\rho_i)x_i^{\HH}Bx_i=\lambda_--\rho_i.$$ Since $\lambda_-B-A\prec0$,
	$\|x_i\|\le\frac{\rho_i-\lambda_-}{\lambda_{\min}(A-\lambda_-B)}\le\frac{\rho_0-\lambda_-}{\lambda_{\min}(A-\lambda_-B)}$,
	which implies that $\|x_i\|$ is bounded.
	To sum up, by Theorem~\ref{thm:convergence:nb=1}, Item~1 holds.

	For Item~2, first, under the assumption $Z_i^{\HH}K^{1/2}F'(\lambda_1)K^{1/2}Z_i=Z_i^{\HH}K^{1/2}BK^{1/2}Z_i\succ0$, it is easy to check that Theorem~3.4 in Golub and Ye \cite{goye:02} still holds, even if the matrix pair $(A,B)$ is definite, rather than restricted to the case that $B\succ0$.
	Then we choose the $m$th Chebyshev polynomial of the first kind as the polynomial $p$ in the theorem.
	Similarly to the discussions by Li \cite[Section~2]{li:06c08}, an upper bound of $\epsilon_m$ in the theorem results.
	Then, together with this theorem, by Theorem~\ref{thm:rate}, Item~2 holds.
\end{proof}

\subsection{Hyperbolic Quadratic Eigenvalue Problems}
As we stated in Section~\ref{sec:intro}, the hyperbolic quadratic polynomial $F(\lambda)=\lambda^2A+\lambda B+C$ for the special case that $\interval=(\lambda_0,+\infty)$, and the smallest positive-type eigenvalue is chosen as what we need, satisfies the assumptions on a generic $F(\lambda)$.
However, with little effort, we know the negative-type eigenvalue or the largest eigenvalue could be transformed into the case mentioned before.
For example, for the largest eigenvalue lying in $\interval=(-\infty,\lambda_0)$, we consider $\what F(\lambda)=F(-\lambda)$ and $\what{\interval}=(-\lambda_0,+\infty)$;
for the largest eigenvalue lying in $\interval=(\lambda_0,+\infty)$, we consider $\what F(\lambda)=-F(-\lambda)$ and $\what{\interval}=(-\infty,\lambda_0)$.

\begin{theorem}\label{thm:HQEP}
	Theorem~\ref{thm:defpair} holds for the hyperbolic quadratic polynomial $$F(\lambda)=\lambda^2 A +\lambda B+C.$$
\end{theorem}
\begin{proof}
	The optimization problem \eqref{eq:iteration-optimize} is
	\begin{equation*} \label{eq:iteration-optimize:HQEP}
		\rho_{i+1}= \rho(Z_iy_i) = \min_{
			\begin{subarray}{c}
				\text{$y^{\HH}Z_i^{\HH}AZ_iy=1$}
			\end{subarray}
		}\rho(Z_iy).
	\end{equation*}
	Using Lagrangian multipliers, it is equivalent to
	\begin{equation*} \label{eq:iteration-optimize:lagrange:HQEP}
		\rho_{i+1}= \min {\cal L}(y,\mu)=\min
		\rho(Z_iy)-\mu(y^{\HH}Z_i^{\HH}AZ_iy-1).
	\end{equation*}
	The minimal point $(y_i,\mu_i)$ must satisfy:
	\begin{subequations}\label{eq:iteration-optimize:langrange:ymu:HQEP}
		\begin{align}
			\frac{\partial {\cal L}(y_i,\mu_i)}{\partial y} &= -2\frac{Z_i^{\HH}r(Z_iy_i)}{\sigma(Z_iy_i)}-2\mu_i Z_i^{\HH}AZ_iy_i=0,\label{eq:iteration-optimize:langrange:y:HQEP}
			\\
			\frac{\partial {\cal L}(y_i,\mu_i)}{\partial \mu} &= y_i^{\HH}Z_i^{\HH}AZ_iy_i-1=0.\label{eq:iteration-optimize:langrange:mu:HQEP}
		\end{align}
	\end{subequations}
	Left multiplying \eqref{eq:iteration-optimize:langrange:y:HQEP} by $y_i^{\HH}$ gives $\mu_i=0$,
	and then $Z_i^{\HH}r(Z_iy_i)=0$.
	Thus, $\frac{\diff \rho(Z_iy_i)}{\diff y}=Z_i^{\HH}\nabla\rho(Z_iy_i)=0$, which means $y_i$ is a stationary point of $\rho(Z_iy)$.
	Besides, under the constraint $x_i^{\HH}Ax_i=1$, $\|x_i\|\le\frac{1}{\lambda_{\min}(A)}$ and then $\|x_i\|$ is bounded.
	To sum up, by Theorem~\ref{thm:convergence:nb=1}, Item~1 holds.

	Item~2 holds by Theorem~\ref{thm:rate}, together with a theorem by Liang and Li\cite[Theorem~9.1]{liangL2015hyperbolic}.
\end{proof}

\section{Numerical Examples}
In the section, we will provide two examples to illustrate the proven convergence rate.
We use the code by Li \cite{li:13:LOBPCGcode} and make small modifications to it to do calculations in the examples below.
All experiments are done in MATLAB R2017a under the Windows 10 Professional 64-bit operating system on a PC with a Intel Core i7-8700 processor at 3.20GHz and 64GB RAM.

\begin{example}[{\cite[Example~12.1]{liangL2015hyperbolic}}]\label{eg:wiresaw1}
	This is the problem \verb|Wiresaw1| in the collection NLEVP \cite{behm:11}.
	It is actually a gyroscopic quadratic eigenvalue problem coming from the vibration analysis of a wiresaw \cite{weka:00}, which we can transform to the following hyperbolic quadratic matrix polynomial:
	\begin{gather*}
		A = \frac{1}{2}I_n, \quad
		C = \frac{(\nu^2-1)\pi^2}{2}\diag(1^2,2^2,\dots,n^2),\\
		B = (b_{ij})\quad \text{with}\quad
		b_{ij}=
		\begin{cases}
			\nu\sqrt{-1}\dfrac{4ij}{i^2-j^2}, & \text{if $i+j$ is odd},\\
			0, & \text{otherwise},
		\end{cases}
	\end{gather*}
	where $\nu$ is a real nonnegative parameter related to the speed of the wire.

	In this example, we use $\LOCG(1,1)$ in Algorithm~\ref{alg:LOBPeCG} with $X_0=\texttt{randn}(n,1)$ for $n=1000$, $\nu=0.1$, with the preconditioner $K=C^{-1}$ to get the smallest positive-type eigenvalue of the problem.
	For the projected problem in every step, the stopping criteria is that the normalized residual is no bigger than $0.1$ or the number of CG steps reaches $10$.
	In Figure~\ref{fig:wiresaw1}, the final approximation is treated as the exact eigenvalue $\lambda_1$, and then: the solid line is the real approximation error; the dash line is the result predicted by (compared with \eqref{eq:thm:defpair-HQEP})
	\[
		\rho_{i+1}-\lambda_1=\frac{2}{\Delta^{2}+\Delta^{-2}}(\rho_i-\lambda_1),\quad
		\Delta=\frac{\sqrt{\kappa}+1}{\sqrt{\kappa}-1}, \quad
		\kappa=\frac{\Gamma}{\gamma}.
	\]
	\begin{figure}[t]
		\centering
		\includegraphics[width=.8\textwidth]{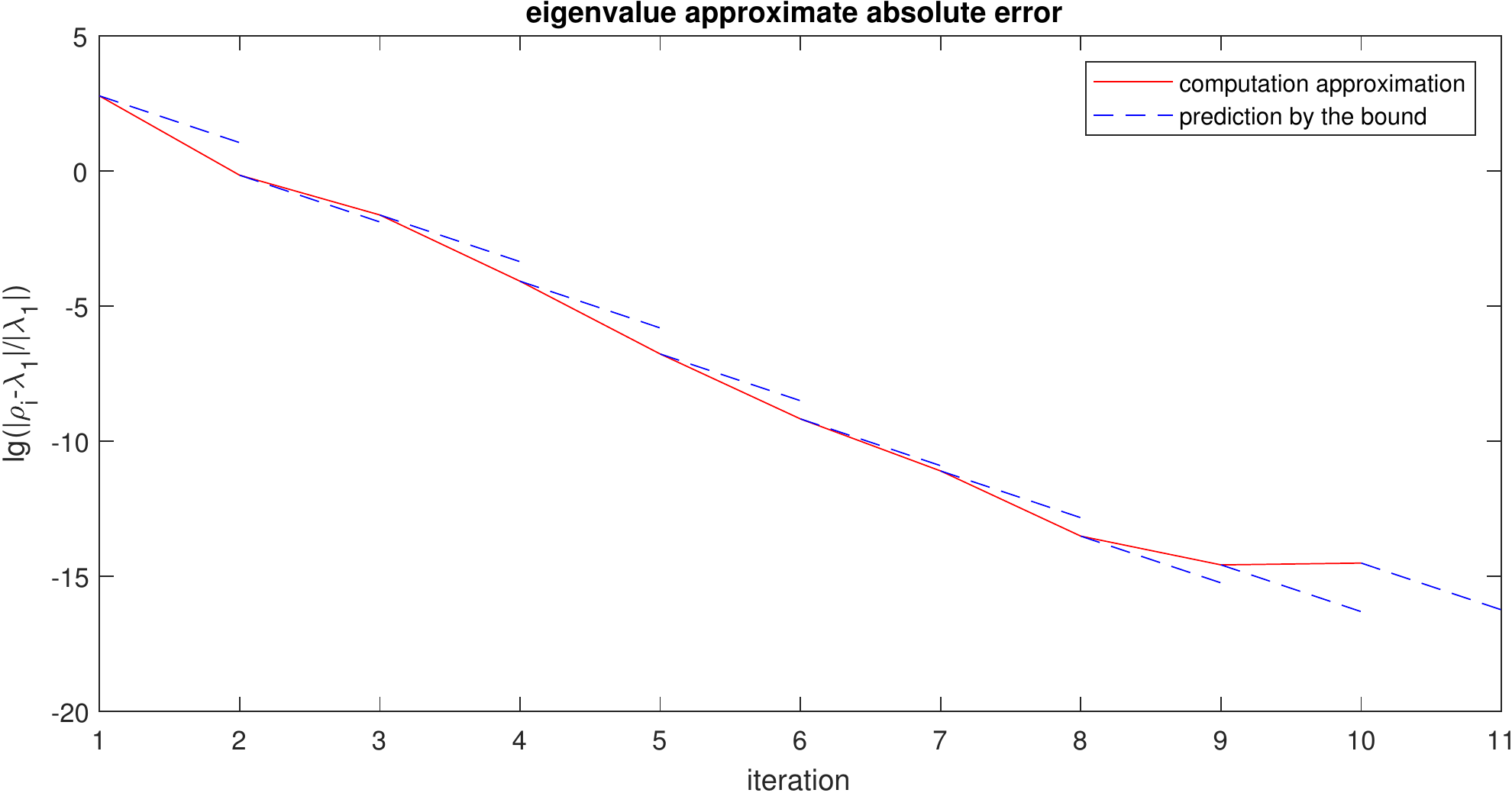}
		\caption{Example~\ref{eg:wiresaw1}: calculation and prediction for $\LOCG(1,1)$.}\label{fig:wiresaw1}
	\end{figure}

	At least we see that in this example, this kind of prediction is appropriate.
\end{example}

\begin{example}\label{eg:gen_hyper2}
	This example is constructed by the MATLAB function \verb|gen_hyper2| in the collection NLEVP \cite{behm:11}. 
	Here, we generate a small-scale problem of size $10$ with eigenvalues $\pm1,\pm2,\dots,\pm10$, and a mid-scale problem of size $1000$ with eigenvalues $\pm1,\pm2,\dots,\pm1000$. The other parameters are chosen randomly.
	Thus, we know the exact eigenvalue $\lambda_1=1$.

	We use different values of $m$ for $\SD(1,m)$ and $\LOCG(1,m)$ to calculate the smallest positive-type eigenvalue, with the preconditioner $K=C^{-1}$.
	For the projected problem in every step, the stopping criteria is that the normalized residual is no bigger than $0.1$ or the number of CG steps reaches $10$.
	\begin{figure}[t]
		\centering
		\subfloat[$n=10$]{
			\includegraphics[width=.8\textwidth]{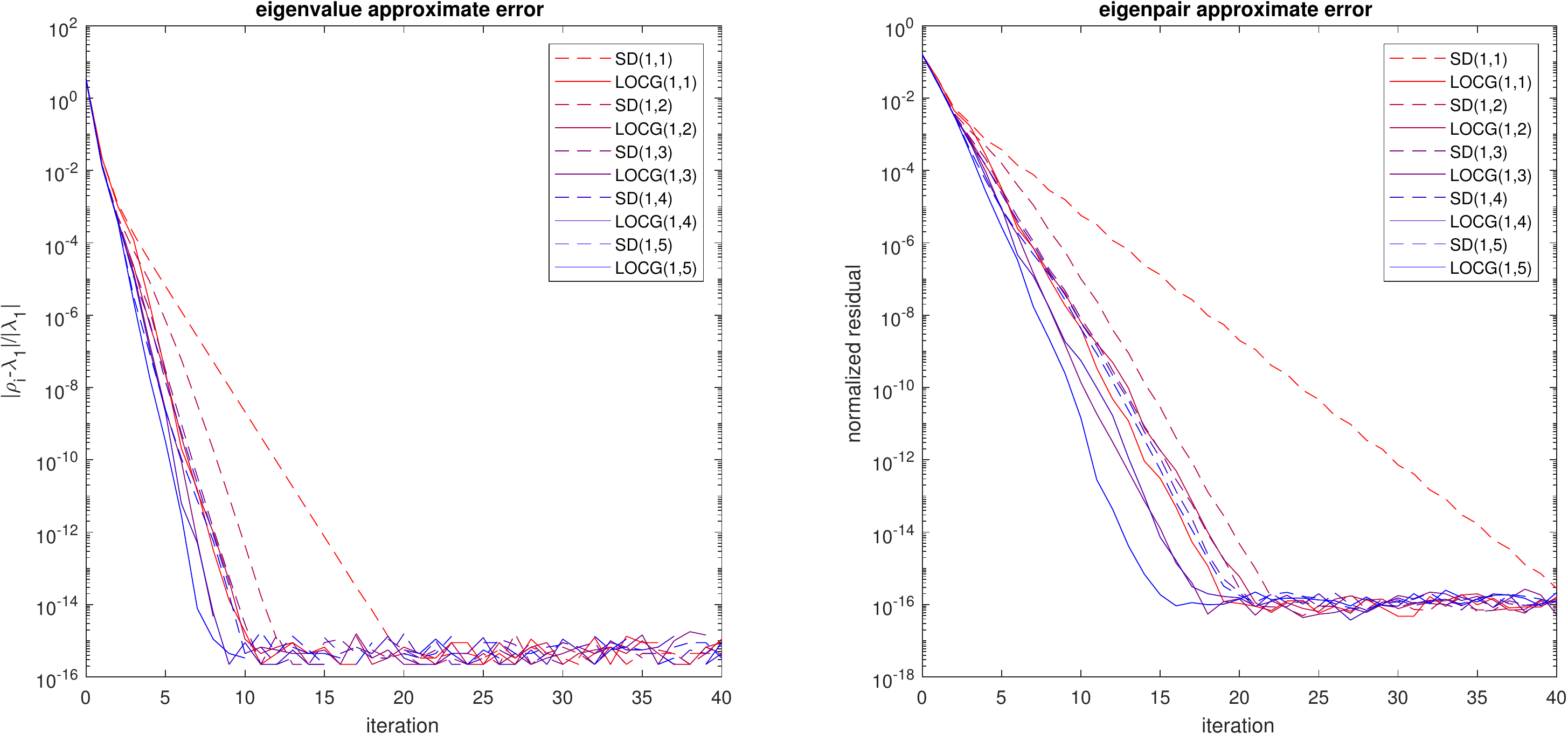}
		}\\
		\subfloat[$n=1000$]{
			\includegraphics[width=.8\textwidth]{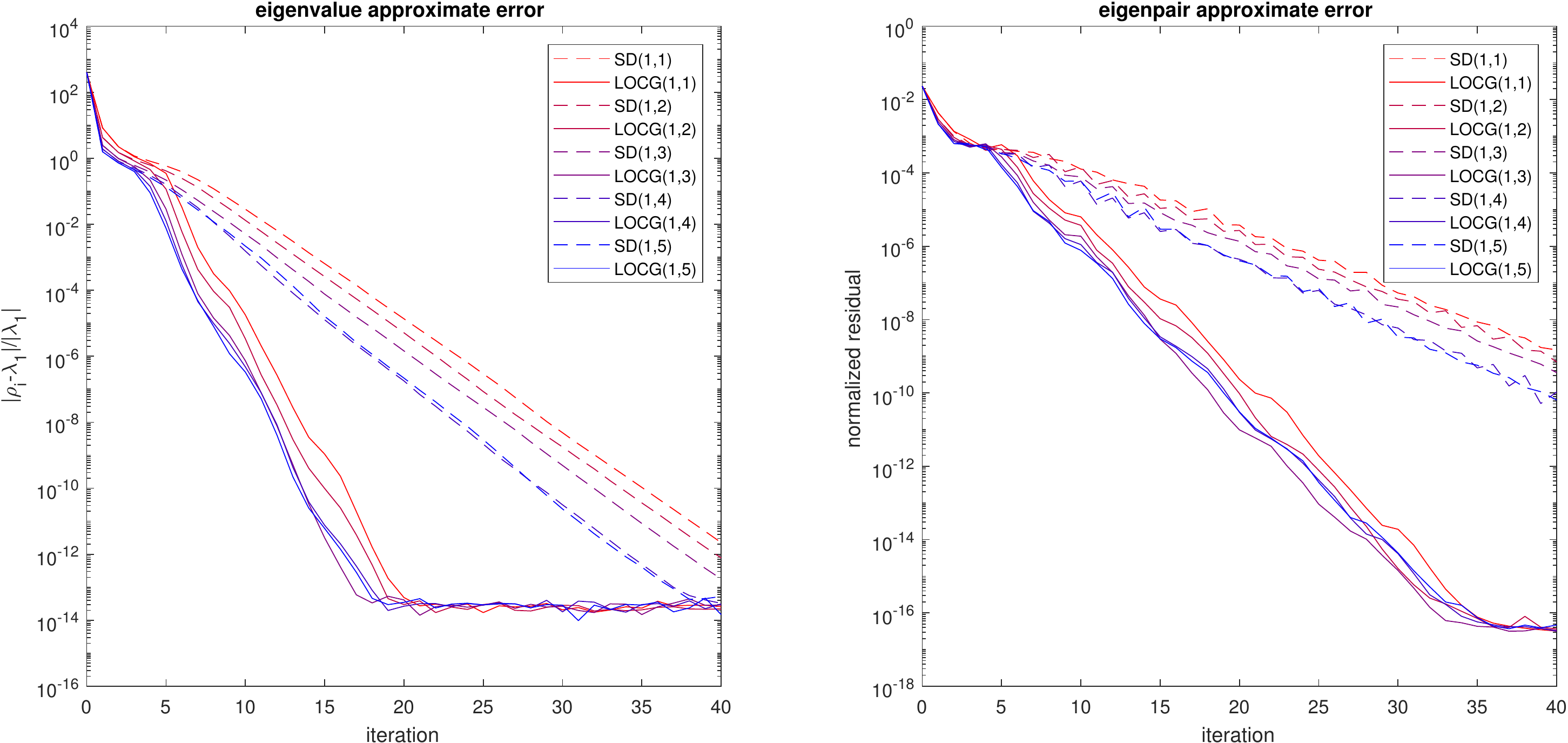}
		}
		\caption{Example~\ref{eg:gen_hyper2}: different dimenions $m$ of Krylov subspaces.}\label{fig:gen_hyper2}
	\end{figure}
	In Figure~\ref{fig:gen_hyper2}, the left figure shows the relative error of the approximations; the right figure shows the normalized residuals
	\[
		\frac{\|Q(\rho_i)x_i\|_2}{(\|A\|_1\rho_i^2+\|B\|_1|\rho_i|+\|C\|_1)\|x_i\|_2}
	\]
	of the approximations.

	In this example, we can see that
    \begin{itemize}
	\item $\LOCG$ is much better than $\SD$, especially for mid/large-scale problems;
	\item increasing the dimension of the Krylov subspace indeed accelerates the convergence to the eigenvalue, though not so significantly;
	\item increasing the dimension of the Krylov subspace perhaps slows down the convergence of the normalized residuals.
    \end{itemize}
	Thus, to balance the computational cost per step and the convergence, maybe the best choice is $\LOCG(1,1)$.
\end{example}

\section{Conclusions}
We have performed the convergence analysis of an extended LOBPCG algorithm for computing the extreme eigenvalue of Hermitian matrix polynomials, including two common instances --- definite matrix pairs and hyperbolic quadratic matrix polynomials.
This analysis was considered out of reach by Kressner et al~\cite[Subsection~3.2]{krps:14:indefinite} or by Liang and Li \cite[Subsection~11.2]{liangL2015hyperbolic} for the vector version of LOBPeCG.
However, it is quite natural to ask whether there exists any kind of convergence analysis for the block version of LOBPeCG.
It is likely that some analogues would hold, but this remains likely to be a difficult and complicated task for future work.

\appendix
\section{A lemma on the inertia property}
For any Hermitian matrix $A$, the inertia of $A$, denoted by $\inertia(A)$,
is a triple of integers which are the number of negative/zero/positive eigenvalues, respectively.

For any real $\lambda$, $F(\lambda)$ is a Hermitian matrix. So we can discuss its inertia, the result is Lemma~\ref{lm:inertia}.
Actually, the lemma is obvious for $\lambda B -A$ when $B\succ0$;
for a definite matrix pair or hyperbolic matrix polynomial $F(\lambda)$, it can be found in many works (see, e.g.~\cite[(0.7)]{vese:10}, \cite[Corollary~2.3.7]{aamm:11:thesis}, and \cite[Section~3]{kovo:13:sylvester}).
\begin{lemma}\label{lm:inertia}
	Given a Hermitian matrix polynomial $F(\lambda)$ satisfying the assumptions at the beginning of Section~\ref{sec:intro}. Then
	\begin{equation}\label{eq:lm:inertia}
		\inertia(F(\mu))= (n-i,0,i),\quad\text{for any}\; \mu\in(\lambda_i,\lambda_{i+1}),
	\end{equation}
	where $i$ is an index to make $\lambda_i<\lambda_{i+1}$.
\end{lemma}
\begin{proof}
	First, for any $\mu\in(\lambda_i,\lambda_{i+1})$, $F(\lambda)$ is nonsingular.
	For $\mu_1$ and $\mu_2$ making $F(\mu_1)$ and $F(\mu)$ have different inertia,
	at least one positive (or negative) eigenvalue of $F(\mu_1)$ has to become a negative (or positive) eigenvalue of $F(\mu_2)$.
	Since the eigenvalues of a matrix, as functions of the matrix entries, are continuous,
	there exists $\mu_3$ between $\mu_1$ and $\mu_2$, such that $F(\mu_3)$ has at least one zero eigenvalue,
	or equivalently, there exists a nonzero vector $x$, such that $F(\mu_3)x=0$.
	This implies $\mu_3$ is an eigenvalue of $F(\lambda)$.
	Thus, for any interval in which no eigenvalue lies, the inertia is invariant.

	Without loss of generality, we assume the eigenvalues are simple. Since $F(\lambda_1-\varepsilon)\prec0$,
	by the continuity of eigenvalues of a matrix, $\inertia(F(\lambda_1))=(n-1,1,0)$.
	Write the corresponding eigenvector of $\lambda_1$ is $u_1$, and then $u_1^{\HH}F(\lambda_1+\varepsilon)u_1>0$.
	Then, also by the continuity, $\inertia(F(\lambda_1+\varepsilon))=(n-1,0,1)$.
	Similarly, we have \eqref{eq:lm:inertia} recursively.
\end{proof}

\section{Claims in the proof of Theorem~2.3}\label{sec:claim-proof}
Before proving the claims, we first establish two bound estimates, which will be used later.

One is that $x_i$ is bounded. Note that
\[
	F_i=F(\lambda_1)+\sum_{k=1}^m\frac{\varepsilon_i^k}{k!}F^{(k)}(\lambda_1), \quad
	\Phi_i=F'(\lambda_1)+\sum_{k=2}^m\frac{\varepsilon_i^{k-1}}{k!}F^{(k)}(\lambda_1).
\]
Since $\varepsilon_i$ is sufficiently small,
$Z_i^{\HH}F'(\rho_i)Z_i\succ0$ implies $Z_i^{\HH}\Phi_iZ_i\succ0,Z_i^{\HH}F'(\lambda_1)Z_i\succ0$.
Let $Q_i=Z_i(Z_i^{\HH}Z_i)^{-1/2}$ be the orthonormal basis of $\range(Z_i)$.
Then
$Q_i^{\HH}F'(\rho_i)Q_i\succ0, Q_i^{\HH}\Phi_iQ_i\succ0,Q_i^{\HH}F'(\lambda_1)Q_i\succ0$.
Write $x_i=Q_i\hat x_i$, and then
\[
	1=x_i^{\HH}\Phi_ix_i=\hat x_i^{\HH}Q_i^{\HH}\Phi_iQ_i\hat x_i\ge \lambda_{\min}(Q_i^{\HH}\Phi_iQ_i)\|\hat x_i\|^2,
\]
which implies
\[
	\|x_i\|^2\le\|\hat x_i\|^2
	\le\frac{1}{\lambda_{\min}(Q_i^{\HH}\Phi_iQ_i)}
	\le\frac{1}{\lambda_{\min}(Q_i^{\HH}F'_iQ_i)}+O(\delta_i)
	.
\]

The other is:
\begin{equation}\label{eq:thm:rate:compare:prf:tFt/tt}
	\text{	
		$-t_i^{\HH}F_it_i\sim t_i^{\HH}\Phi_it_i\sim t_i^{\HH}t_i$,
		for any $t_i=Q_i\hat t_i\ne0$ satisfying $t_i^{\HH}\Phi_ix_i=0$.
	}
\end{equation}
In fact, since $Q_i^{\HH}\Phi_iQ_i\succ0$, $t_i^{\HH}\Phi_it_i=\hat t^{\HH}Q_i^{\HH}\Phi_iQ_i\hat t\sim \hat t_i^{\HH}\hat t_i\sim t_i^{\HH}t_i$.
For the rest, since $x_i^{\HH}\Phi_ix_i=1,x_i^{\HH}F_ix_i=0$,
using the min-max principles \eqref{eq:trace:min} for the definite matrix pair $(-Q_i^{\HH}F_iQ_i, Q_i^{\HH}\Phi_iQ_i)$,
\begin{align*}
	-\frac{\hat t_i^{\HH}Q_i^{\HH}F_iQ_i\hat t_i }{\hat t_i^{\HH}Q_i^{\HH}\Phi_iQ_i\hat t_i}
	&= \frac{\hat t_i^{\HH}(-Q_i^{\HH}F_iQ_i)\hat t_i }{\hat t_i^{\HH}Q_i^{\HH}\Phi_iQ_i\hat t_i}+\frac{\hat x_i^{\HH}(-Q_i^{\HH}F_iQ_i)\hat x_i }{\hat x_i^{\HH}Q_i^{\HH}\Phi_iQ_i\hat x_i}
	\\&\ge \lambda_{\min}(-[Q_i^{\HH}\Phi_iQ_i]^{-1/2}Q_i^{\HH}F_iQ_i[Q_i^{\HH}\Phi_iQ_i]^{-1/2})+\lambda_{\min}^{(2)}(-[Q_i^{\HH}\Phi_iQ_i]^{-1/2}Q_i^{\HH}F_iQ_i[Q_i^{\HH}\Phi_iQ_i]^{-1/2})
	\\&= 0+\lambda_{\min}^{(2)}(-[Q_i^{\HH}F'(\lambda_1)Q_i]^{-1/2}Q_i^{\HH}F(\lambda_1)Q_i[Q_i^{\HH}F'(\lambda_1)Q_i]^{-1/2})+O(\varepsilon_i)
	.
\end{align*}
By \eqref{eq:courant-fischer:minmax},
\begin{align*}
	&\hspace*{-2cm}\lambda_{\min}^{(2)}(-[Q_i^{\HH}F'(\lambda_1)Q_i]^{-1/2}Q_i^{\HH}F(\lambda_1)Q_i[Q_i^{\HH}F'(\lambda_1)Q_i]^{-1/2})
	\\&=\min_{\dim \mathcal{U}=2}\max_{u\in \mathcal{U}}\frac{-u^{\HH}[Q_i^{\HH}F'(\lambda_1)Q_i]^{-1/2}Q_i^{\HH}F(\lambda_1)Q_i[Q_i^{\HH}F'(\lambda_1)Q_i]^{-1/2}u}{u^{\HH}u}
	\\&\qquad\qquad\text{ (write $v=[Q_i^{\HH}F'(\lambda_1)Q_i]^{-1/2}u$, and then $u=[Q_i^{\HH}F'(\lambda_1)Q_i]^{1/2}v$) }
	\\&=\min_{\dim \mathcal{V}=2}\max_{v\in \mathcal{V}}\frac{-v^{\HH}Q_i^{\HH}F(\lambda_1)Q_iv}{v^{\HH}Q_i^{\HH}Q_iv}\frac{v^{\HH}v}{v^{\HH}Q_i^{\HH}F'(\lambda_1)Q_iv}
	\qquad\text{(write $w=Q_i^{\HH}v$)}
	\\&=\min_{\dim \mathcal{W}=2}\max_{w\in \mathcal{W}}\frac{-w^{\HH}F(\lambda_1)w}{w^{\HH}w}\frac{v^{\HH}v}{v^{\HH}Q_i^{\HH}F'(\lambda_1)Q_iv}
	\\&\ge\min_{\dim \mathcal{V}=2}\max_{v\in \mathcal{V}}\frac{-v^{\HH}F(\lambda_1)v}{v^{\HH}v}\frac{1}{\lambda_{\max}(Q_i^{\HH}F'(\lambda_1)Q_i)}
	\\&=\frac{\lambda_{\min}^{(2)}(-F(\lambda_1))}{\lambda_{\max}(Q_i^{\HH}F'(\lambda_1)Q_i)}
	\\&\ge\frac{-\lambda_{\max}^{(2)}(F(\lambda_1))}{\lambda_{\max}(F'(\lambda_1))}
	=:\omega
	.
\end{align*}
Thus,
\begin{equation}\label{eq:thm:rate:compare:prf:omega}
	\frac{- t_i^{\HH}F_it_i }{t_i^{\HH}\Phi_it_i}
	=-\frac{\hat t_i^{\HH}Q_i^{\HH}F_iQ_i\hat t_i }{\hat t_i^{\HH}Q_i^{\HH}\Phi_iQ_i\hat t_i}
	\ge\omega+O(\varepsilon_i)>0.
\end{equation}
On the other hand, $-t_i^{\HH}F_it_i\le\|F_i\|t_i^{\HH}t_i\sim t_i^{\HH}\Phi_it_i$. In total, $-t_i^{\HH}F_it_i\sim t_i^{\HH}\Phi_it_i$.

Now we can begin to prove those claims.
\begin{proof}[Proof \emph{of \eqref{eq:thm:rate:compare:prf:claim:inv}}]
	Note that $\range(S_i)\subset\range(Z_i)$ and $S_i^{\HH}\Phi_ix_i=0$.
	By \eqref{eq:thm:rate:compare:prf:omega}, $-\hat t_i^{\HH}S_i^{\HH}F_iS_i\hat t_i\ge(\omega+O(\varepsilon_i))\hat t_i^{\HH}S_i^{\HH}\Phi_iS_i\hat t_i$.
	Hence
	\[
		\lambda_{\min}(-S_i^{\HH}F_iS_i)
		\ge(\omega+O(\varepsilon_i))\lambda_{\min}(S_i^{\HH}\Phi_iS_i)
		\ge(\omega+O(\varepsilon_i))\lambda_{\min}(Q_i^{\HH}\Phi_iQ_i)\lambda_{\min}(S_i^{\HH}S_i)
		>0.
	\]
	Note that $S_i^{\HH}F_{i+1}S_i=S_i^{\HH}F_iS_i-\delta_iS_i\Phi_iS_i$.
	It is clear that $
	\lambda_{\min}(-S_i^{\HH}F_{i+1}S_i)\ge \omega\lambda_{\min}(Q_i^{\HH}F'(\lambda_1)Q_i)+O(\varepsilon_i)>0$,
	which implies that $S_i^{\HH}F_{i+1}S_i$ is nonsingular.
	It is similar that $S_i^{\HH}F_{i+1}^{\eSD}S_i$ is nonsingular.
\end{proof}

\begin{proof}[Proof \emph{of \eqref{eq:thm:rate:compare:prf:claim:kappa4:rFr/rr}}]
	Since $(I-P_i)(I-T_i)r_i\in \range(Z_i)$ and $x_i^{\HH}\Phi_i(I-P_i)(I-T_i)r_i=0$,
	by \eqref{eq:thm:rate:compare:prf:tFt/tt},
	\[
		-r_i^{\HH}\check F_i(I-T_i)r_i=-r_i^{\HH}(I-T_i^{\HH})\check F_i(I-T_i)r_i\sim r_i^{\HH}r_i.
	\]
	For the rest,
	let $\rho_{i+1}^{\SD}$ be the minimal value of $\rho(x)$ in the subspace $\subspan\{x_i,r_i\}$,
	then
	\[
		\delta_i^{\SD} = -\frac{|r_i^{\HH}r_i|^2}{[x_i^{\HH}\Phi_i^{\SD}x_i][r_i^{\HH}\check F_i^{\SD}r_i]}
		\quad\Rightarrow\quad
		r_i^{\HH}r_i=-\delta_i^{\SD}[x_i^{\HH}\Phi_i^{\SD}x_i]\frac{r_i^{\HH}\check F_i^{\SD}r_i}{r_i^{\HH}r_i}=O(\delta_i).
		\qedhere
	\]
\end{proof}

\begin{proof}[Proof \emph{of \eqref{eq:thm:rate:compare:prf:claim:kappa23}}]
	Consider $\kappa_2$.
	\begin{align*}
		\kappa_2
		&= \frac{(s_i^{\eSD})^{\HH}\check F_is_i^{\eSD}}{(s_i^{\eSD})^{\HH}\check F_i^{\eSD}s_i^{\eSD}}
		= \frac{(s_i^{\eSD})^{\HH}(I-P_i^{\HH})F_{i+1}(I-P_i)s_i^{\eSD}}{(s_i^{\eSD})^{\HH}(I-(P_i^{\eSD})^{\HH})F_{i+1}^{\eSD}(I-P_i^{\eSD})s_i^{\eSD}}
		\\&= \frac{(s_i^{\eSD})^{\HH}(I-P_i^{\HH})F_i(I-P_i)s_i^{\eSD}-\delta_{i+1}(s_i^{\eSD})^{\HH}(I-P_i^{\HH})\Phi_i(I-P_i)s_i^{\eSD}}{(s_i^{\eSD})^{\HH}(I-(P_i^{\eSD})^{\HH})F_i(I-P_i^{\eSD})s_i^{\eSD}-\delta_{i+1}^{\eSD}(s_i^{\eSD})^{\HH}(I-(P_i^{\eSD})^{\HH})\Phi_i^{\eSD}(I-P_i^{\eSD})s_i^{\eSD}}
		.
	\end{align*}
	Since $(I-P_i^{\HH})s_i^{\eSD}\in\range(Z_i)$ and $(I-P_i^{\HH})\Phi_ix_i=0$, by \eqref{eq:thm:rate:compare:prf:tFt/tt},
	\[
		-(s_i^{\eSD})^{\HH}(I-P_i^{\HH})F_i(I-P_i)s_i^{\eSD} \sim (s_i^{\eSD})^{\HH}(I-P_i^{\HH})\Phi_i(I-P_i)s_i^{\eSD};
	\]
	since $(I-(P_i^{\eSD})^{\HH})s_i^{\eSD}\in\range(Z_i)$ and $(I-(P_i^{\eSD})^{\HH})\Phi_i^{\eSD}x_i=0$, then similarly to \eqref{eq:thm:rate:compare:prf:tFt/tt}, we have
	\[
		-(s_i^{\eSD})^{\HH}(I-(P_i^{\eSD})^{\HH})F_i(I-P_i^{\eSD})s_i^{\eSD} \sim (s_i^{\eSD})^{\HH}(I-(P_i^{\eSD})^{\HH})\Phi_i^{\eSD}(I-P_i^{\eSD})s_i^{\eSD}.
	\]
	Thus
	\[
		\kappa_2
		= \frac{[1+O(\delta_i)](s_i^{\eSD})^{\HH}(I-P_i^{\HH})F_i(I-P_i)s_i^{\eSD}}{[1+O(\delta_i^{\eSD})](s_i^{\eSD})^{\HH}(I-(P_i^{\eSD})^{\HH})F_i(I-P_i^{\eSD})s_i^{\eSD}}.
	\]
	Note that
	\begin{align*}
		0\ge (s_i^{\eSD})^{\HH}(I-P_i^{\HH})F_i(I-P_i)s_i^{\eSD}
		&=(s_i^{\eSD})^{\HH}F_is_i^{\eSD}-2\Re (s_i^{\eSD})^{\HH}\Phi_ix_ix_i^{\HH}F_is_i^{\eSD}+(s_i^{\eSD})^{\HH}\Phi_ix_ix_i^{\HH}F_ix_ix_i^{\HH}\Phi s_i^{\eSD}
		\\&=(s_i^{\eSD})^{\HH}F_is_i^{\eSD}-2r_i^{\HH}r_i\Re (s_i^{\eSD})^{\HH}\Phi_ix_i
		,
	\end{align*}
	and a similar expansion of $(s_i^{\eSD})^{\HH}(I-(P_i^{\eSD})^{\HH})F_i(I-P_i^{\eSD})s_i^{\eSD}$ holds.
	Then
	\begin{align*}
		\kappa_2
		&= [1+O(\delta_i)]\frac{(s_i^{\eSD})^{\HH}F_is_i^{\eSD}-2r_i^{\HH}r_i\Re (s_i^{\eSD})^{\HH}\Phi_ix_i}{(s_i^{\eSD})^{\HH}F_is_i^{\eSD}-2r_i^{\HH}r_i\Re (s_i^{\eSD})^{\HH}\Phi_i^{\eSD}x_i}
		\\&= [1+O(\delta_i)]\left[1+\frac{2r_i^{\HH}r_i\Re (s_i^{\eSD})^{\HH}[\Phi_i^{\eSD}-\Phi_i]x_i}{(s_i^{\eSD})^{\HH}F_is_i^{\eSD}-2r_i^{\HH}r_i\Re (s_i^{\eSD})^{\HH}\Phi_i^{\eSD}x_i} \right]
		\\&= [1+O(\delta_i)]\left[1+\frac{2r_i^{\HH}r_i\Re (s_i^{\eSD})^{\HH}[\Phi_i^{\eSD}-\Phi_i]x_i}{[1+O(\delta_i^{\eSD})](s_i^{\eSD})^{\HH}\check F_i^{\eSD}s_i^{\eSD}} \right]
		.
	\end{align*}
	It is easy to see that
	\[
		(s_i^{\eSD})^{\HH}\check F_i^{\eSD}s_i^{\eSD}
		=r_i^{\HH}(I-(T^{\eSD})^{\HH})\check F_i^{\eSD}(I-T^{\eSD})r_i
		=r_i^{\HH}\check F_i^{\eSD}r_i-r_i^{\HH}\check F_i^{\eSD}S_i(S_i^{\HH}\check F_i^{\eSD}S_i)^{-1}S_i^{\HH}\check F_i^{\eSD}r_i
		.
	\]
	Similarly to the proof of \eqref{eq:thm:rate:compare:prf:claim:inv}, we know
	$-\begin{bmatrix}
		S_i & r_i
		\end{bmatrix}^{\HH}\check F_i^{\eSD}\begin{bmatrix}
		S_i & r_i
		\end{bmatrix}=-\begin{bmatrix}
		S_i & (I-P_i^{\eSD})r_i
		\end{bmatrix}^{\HH}F_{i+1}^{\eSD}\begin{bmatrix}
		S_i & (I-P_i^{\eSD})r_i
	\end{bmatrix}$ is positive definite.
	Thus, since $r_i^{\HH}S_i=0$, by a matrix version of the Wielandt inequality (see Wang and Ip \cite[Theorem~1]{wangI1999matrix}),
	\[
		-r_i^{\HH}\check F_i^{\eSD}S_i(S_i^{\HH}\check F_i^{\eSD}S_i)^{-1}S_i^{\HH}\check F_i^{\eSD}r_i\le -[\chi+O(\varepsilon_i)] r_i^{\HH}\check F_i^{\eSD} r_i,
		\qquad \chi=\left(\frac{\lambda_{\max}(-F(\lambda_1))-\lambda_{\min}^{(2)}(-F(\lambda_1))}{\lambda_{\max}(-F(\lambda_1))+\lambda_{\min}^{(2)}(-F(\lambda_1))}\right)^2.
	\]
	which gives $-(s_i^{\eSD})^{\HH}\check F_i^{\eSD}s_i^{\eSD}\sim -r_i^{\HH}\check F_i^{\eSD}r_i.$
	Note that by \eqref{eq:thm:rate:compare:prf:tFt/tt}, $-r_i^{\HH}\check F_i^{\eSD} r_i\sim r_i^{\HH}r_i, -(s_i^{\eSD})^{\HH}\check F_i^{\eSD}s_i^{\eSD}\sim (s_i^{\eSD})^{\HH}s_i^{\eSD}$.
	Thus,
	\begin{equation}\label{eq:thm:rate:compare:prf:sFs}
		-r_i^{\HH}\check F_i^{\eSD} r_i\sim r_i^{\HH}r_i\sim (s_i^{\eSD})^{\HH}s_i\sim -(s_i^{\eSD})^{\HH}\check F_i^{\eSD}s_i^{\eSD},
	\end{equation}
	and
	\[
		\kappa_2 = [1+O(\delta_i)]\left(1+O(1)\Re (s_i^{\eSD})^{\HH}[\Phi_i^{\eSD}-\Phi_i]x_i \right)
		.
	\]
	Noticing that
	\[
		\left|(s_i^{\eSD})^{\HH}[\Phi_i^{\eSD}-\Phi_i]x_i\right|\le\|s_i\|(\delta_i-\delta_i^{\eSD})\left[\|F''(\rho_i)\|+O(\delta_i)\right]\|x_i\|=O(\delta_i^{3/2}),
	\]
	we have
	\[
		\kappa_2 = [1+O(\delta_i)]\left(1+O(\delta_i^{3/2})\right)=1+O(\delta_i).
	\]

	Consider $\kappa_3$.
	By the Sherman-Morrison-Woodbury formula, letting $D_i=\check F_i^{\eSD}-\check F_i$,
	\begin{align*}
		e_i
		&= S_i(S_i^{\HH}\check F_i^{\eSD}S_i)^{-1}S_i^{\HH}\check F_i^{\eSD}r_i - S_i(S_i^{\HH}\check F_iS_i)^{-1}S_i^{\HH}\check F_ir_i ,
		\\&= S_i\left[(S_i^{\HH}\check F_i^{\eSD}S_i)^{-1}S_i^{\HH}\check F_i^{\eSD} - (S_i^{\HH}\check F_iS_i)^{-1}S_i^{\HH}\check F_i\right]r_i
		\\&= S_i\left[ \left((S_i^{\HH}\check F_iS_i)^{-1}-(S_i^{\HH}\check F_iS_i)^{-1}S_i^{\HH}D_iS_i(S_i^{\HH}\check F_i^{\eSD}S_i)^{-1}\right)S_i^{\HH}\check F_i^{\eSD} - (S_i^{\HH}\check F_iS_i)^{-1}S_i^{\HH}\check F_i\right]r_i
		\\&= S_i\left[ (S_i^{\HH}\check F_iS_i)^{-1}S_i^{\HH}D_i-(S_i^{\HH}\check F_iS_i)^{-1}S_i^{\HH}D_iS_i(S_i^{\HH}\check F_i^{\eSD}S_i)^{-1}S_i^{\HH}\check F_i^{\eSD}\right]r_i
		\\&= S_i(S_i^{\HH}\check F_iS_i)^{-1}S_i^{\HH}D_i\left[I-S_i(S_i^{\HH}\check F_i^{\eSD}S_i)^{-1}S_i^{\HH}\check F_i^{\eSD}\right]r_i
		\\&= S_i(S_i^{\HH}\check F_iS_i)^{-1}S_i^{\HH}D_is_i^{\eSD}.
	\end{align*}
	Since
	$
	S_i(S_i^{\HH}\check F_iS_i)^{-1}S_i^{\HH}\check F_i(s_i^{\eSD}+e_i)=T_i(I-T_i)r_i=0,
	$
	we have $e_i^{\HH}\check F_i(s_i^{\eSD}+e_i)=0$.
	Thus,
	\[
		\kappa_3
		= \frac{(s_i^{\eSD}+e_i)^{\HH}\check F_i(s_i^{\eSD}+e_i)}{(s_i^{\eSD})^{\HH}\check F_is_i^{\eSD}}
		= 1-\frac{e_i^{\HH}\check F_ie_i}{(s_i^{\eSD})^{\HH}\check F_is_i^{\eSD}}
		= 1-\frac{(s_i^{\eSD})^{\HH}D_iS_i(S_i^{\HH}\check F_iS_i)^{-1}S_i^{\HH}D_is_i^{\eSD}}{(s_i^{\eSD})^{\HH}\check F_is_i^{\eSD}}
		.
	\]
	First we estimate $S_i(S_i^{\HH}\check F_iS_i)^{-1}S_i^{\HH}$.
	Let $S_i=Q_SR_S$ be its QR factorization, and then
	\[
		S_i(S_i^{\HH}\check F_iS_i)^{-1}S_i^{\HH}
		= Q_S(Q_S^{\HH}\check F_iQ_S)^{-1}Q_S^{\HH}
		= Q_S\left(Q_S^{\HH}(I-P_i^{\HH})F_{i+1}(I-P_i)Q_S\right)^{-1}Q_S^{\HH}
		.
	\]
	Since $\range((I-P_i)Q_S)\subset \range(Z_i)$, similarly to the proof of \eqref{eq:thm:rate:compare:prf:claim:inv},
	we have
	\begin{equation}\label{eq:thm:rate:compare:prf:T}
		\| S_i(S_i^{\HH}\check F_iS_i)^{-1}S_i^{\HH}\|\le \frac{1}{\omega\lambda_{\min}(Q_S^{\HH}F'(\lambda_1)Q_S)+O(\varepsilon_i)} .
	\end{equation}
	Then turn to $D_i$.
	Noticing $(I-P_i)Q_S=Q_S$,
	\begin{align*}
		Q_S^{\HH}D_i
		&=Q_S^{\HH}\left[ (I-(P_i^{\eSD})^{\HH})F_{i+1}^{\eSD}(I-P_i^{\eSD})-(I-P_i^{\HH})F_{i+1}(I-P_i) \right]
		\\&=Q_S^{\HH}\left[(P_i^{\HH}-(P_i^{\eSD})^{\HH})F_{i+1}^{\eSD}(I-P_i^{\eSD})+ F_{i+1}^{\eSD}(I-P_i^{\eSD})-F_{i+1}(I-P_i) \right]
		\\&=Q_S^{\HH}\left[ (P_i^{\HH}-P_i^{\eSD})^{\HH}F_{i+1}^{\eSD}(I-P_i^{\eSD})+F_{i+1}^{\eSD}(P_i-P_i^{\eSD})+(F_{i+1}^{\eSD}-F_{i+1})(I-P_i) \right]
		\\&=Q_S^{\HH}\left[ (\Phi_i-\Phi_i^{\eSD})x_ix_i^{\HH}F_{i+1}^{\eSD}(I-P_i^{\eSD})+F_{i+1}^{\eSD}x_ix_i^{\HH}(\Phi_i-\Phi_i^{\eSD})+(F_{i+1}^{\eSD}-F_{i+1})(I-P_i) \right]
	\end{align*}
	and then
	\[
		\|Q_S^{\HH}D_i\|
		\le(\delta_i-\delta_i^{\eSD})\left[ \left(\|F''(\rho_i)\|+O(\delta_i)\right)\|x_i\|^2\|F_{i+1}^{\eSD}\|\left( \|I-P_i^{\eSD}\| +1\right)+\left(\|F'(\rho_i)+O(\delta_i)\right)\|I-P_i\| \right]
		=O(\delta_i)
		.
	\]
	Thus, to sum up, together with \eqref{eq:thm:rate:compare:prf:sFs},
	\begin{align*}
		\kappa_3
		&= 1-\frac{(s_i^{\eSD})^{\HH}D_iQ_SQ_S^{\HH}S_i(S_i^{\HH}\check F_iS_i)^{-1}S_i^{\HH}Q_S Q_S^{\HH}D_is_i^{\eSD}}{(s_i^{\eSD})^{\HH}\check F_is_i^{\eSD}}
		\\&= 1+\frac{O(1)\|Q_S^{\HH}D_i\|^2\|s_i^{\eSD}\|^2}{\|s_i^{\eSD}\|^2}
		= 1-O(\delta_i^2)
		.
		\qedhere
	\end{align*}
\end{proof}

\begin{proof}[Proof \emph{of \eqref{eq:thm:rate:compare:prf:claim:kappa4:dd}}]
	Since $\check F_ix_{i-1}=\check F_i(x_i-d_{i-1})=-\check F_id_{i-1}$
	and 
	$\check F_i(I-T_i)=(I-T_i^{\HH})\check F_i$,
	\[
		\tau_i=x_{i-1}^{\HH}\check F_i(I-T_i)x_{i-1}=d_{i-1}^{\HH}\check F_i(I-T_i)d_{i-1}.
	\]
	First
	\begin{equation} \label{eq:thm:rate:compare:prf:F}
		\begin{aligned}[b]
			\check F_i
			&= (I-P_i^{\HH})F_{i+1}(I-P_i)
			\\&= P_i^{\HH}F_{i+1}P_i -P_i^{\HH}F_{i+1}-F_{i+1}P_i+F_{i+1}- F_i + F_i
			\\&= F_{i+1}- F_i+\Phi_ix_ix_i^{\HH}F_{i+1}x_ix_i^{\HH}\Phi_i -\Phi_ix_ix_i^{\HH}F_{i+1}-F_{i+1}x_ix_i^{\HH}\Phi_i + F_i
			\\&= -\delta_i\Phi_i-\delta_i\Phi_ix_ix_i^{\HH}\Phi_i -\Phi_ix_ix_i^{\HH}(F_i-\delta_i\Phi_i)-(F_i-\delta_i\Phi_i)x_ix_i^{\HH}\Phi_i + F_i
			\\&=  F_i - \Phi_ix_ir_i^{\HH} - r_ix_i^{\HH}\Phi_i    -\delta_i\Phi_i[I-x_ix_i^{\HH}\Phi_i]	
			.
		\end{aligned}
	\end{equation}
	Since $r_i^{\HH}d_{i-1}=0$ by \eqref{eq:thm:linesearch:r_opt-perp}, 
	\begin{align*}
		d_{i-1}^{\HH}\check F_id_{i-1}
		&=  d_{i-1}^{\HH}F_id_{i-1} - d_{i-1}^{\HH}\Phi_ix_ir_i^{\HH}d_{i-1} - d_{i-1}^{\HH}r_ix_i^{\HH}\Phi_id_{i-1} -\delta_id_{i-1}^{\HH}\Phi_i(I-P_i)d_{i-1}
		\\&=  d_{i-1}^{\HH}F_id_{i-1}-\delta_id_{i-1}^{\HH}\Phi_i(I-P_i)d_{i-1}
		\\&= d_{i-1}^{\HH}F_id_{i-1}+O(\delta_i)\|d_{i-1}\|^2
		.
	\end{align*}
	Then, noticing that
	$x_{i-1}^{\HH}\Phi_{i-1}d_{i-1}=x_{i-1}^{\HH}\Phi_{i-1}(I-P_{i-1})d_{i-1}=0$,
	by \eqref{eq:thm:linesearch:rho}, 
	\begin{align*}
		d_{i-1}^{\HH}F_id_{i-1}
		= -\delta_{i-1} x_{i-1}^{\HH}\Phi_{i-1}x_{i-1}
		&= -\delta_{i-1} (x_i-d_{i-1})^{\HH}\Phi_{i-1}(x_i-d_{i-1})
		\\&= -\delta_{i-1} (x_i^{\HH}\Phi_{i-1}x_i-d_{i-1}^{\HH}\Phi_{i-1}d_{i-1})
		\\&= -\delta_{i-1} (x_i^{\HH}\Phi_ix_i-d_{i-1}^{\HH}\Phi_{i-1}d_{i-1}+O(\delta_{i-1}))
		\\&= -\delta_{i-1} (1-d_{i-1}^{\HH}\Phi_{i-1}d_{i-1}+O(\delta_{i-1}))
		.
	\end{align*}
	Similarly to \eqref{eq:thm:rate:compare:prf:tFt/tt},
	$-d_{i-1}^{\HH}F_id_{i-1}\sim d_{i-1}^{\HH}\Phi_{i-1}d_{i-1}$,
	which implies
	\[
		d_{i-1}^{\HH}F_id_{i-1}
		= -\frac{\delta_{i-1}}{1+O(\delta_{i-1})}+O(\delta_{i-1}^2)
		= -\delta_{i-1}+O(\delta_{i-1}^2)
		,
	\]
	and $\delta_{i-1}\sim d_{i-1}^{\HH}\Phi_{i-1}d_{i-1}\sim d_{i-1}^{\HH}d_{i-1}$.
	Thus
	\begin{equation} \label{eq:thm:rate:compare:prf:dFd}
		d_{i-1}^{\HH}\check F_id_{i-1}= -\delta_{i-1}[1+O(\delta_{i-1})+O(\delta_i)].
	\end{equation}
	Then, by \eqref{eq:thm:rate:compare:prf:F},
	\begin{align*}
		d_{i-1}^{\HH}\check F_iT_id_{i-1}
		&=  d_{i-1}^{\HH}F_iT_id_{i-1} - d_{i-1}^{\HH}\Phi_ix_ir_i^{\HH}T_id_{i-1} - d_{i-1}^{\HH}r_ix_i^{\HH}\Phi_iT_id_{i-1}  -\delta_id_{i-1}^{\HH}\Phi_i[I-x_ix_i^{\HH}\Phi_i]T_id_{i-1}
		\\&=   d_{i-1}^{\HH}F_iT_id_{i-1}-\delta_id_{i-1}^{\HH}\Phi_i[I-x_ix_i^{\HH}\Phi_i]T_id_{i-1}.
	\end{align*}
	Since $\|T_i\|\le\|S_i(S_i^{\HH}\check F_iS_i)^{-1}S_i^{\HH}\|\|F_i\|\|I-P_i\|^2$, by \eqref{eq:thm:rate:compare:prf:T},
	\[
		d_{i-1}^{\HH}\check F_iT_id_{i-1}
		=  d_{i-1}^{\HH}F_iT_id_{i-1}+O(\delta_i\delta_{i-1})
		.
	\]
	Then, also using \eqref{eq:thm:rate:compare:prf:T},
	\begin{align*}
		d_{i-1}^{\HH}F_iT_id_{i-1}
		&= [x_i-x_{i-1}]^{\HH}F_iS_i(S_i^{\HH}\check F_iS_i)^{-1}S_i^{\HH}F_{i+1}(I-P_i)[x_i-x_{i-1}]
		\\&= [x_i-x_{i-1}]^{\HH}F_iS_i(S_i^{\HH}\check F_iS_i)^{-1}S_i^{\HH}F_{i+1}[x_i(x_i^{\HH}\Phi_ix_{i-1})-x_{i-1}]
		\\&= [r_i-(F_{i-1}-\delta_{i-1}\Phi_{i-1})x_{i-1}]^{\HH}S_i(S_i^{\HH}\check F_iS_i)^{-1}S_i^{\HH}
		[(F_i-\delta_i\Phi_i)x_i(x_i^{\HH}\Phi_ix_{i-1})-(F_{i-1}-\delta_{i-1}\Phi_{i-1}-\delta_i\Phi_i)x_{i-1}]
		\\&= [r_i^{\HH}-r_{i-1}^{\HH}+\delta_{i-1}x_{i-1}^{\HH}\Phi_{i-1}]S_i(S_i^{\HH}\check F_iS_i)^{-1}S_i^{\HH}
		[(r_i-\delta_i\Phi_ix_i)(1+O(\delta_{i-1}^{1/2}))-r_{i-1}+\delta_{i-1}\Phi_{i-1}x_{i-1}+\delta_i\Phi_ix_{i-1}]
		\\&= [-r_{i-1}^{\HH}+\delta_{i-1}x_{i-1}^{\HH}\Phi_{i-1}]S_i(S_i^{\HH}\check F_iS_i)^{-1}S_i^{\HH}
		[-r_{i-1}+\delta_{i-1}\Phi_{i-1}x_{i-1}-\delta_i\Phi_i(d_{i-1}+x_iO(\delta_{i-1}^{1/2}))]
		\\&=r_{i-1}^{\HH}S_i(S_i^{\HH}\check F_iS_i)^{-1}S_i^{\HH}r_{i-1}+O(\delta_{i-1}^{3/2})
		.
	\end{align*}
	Since $r_i^{\HH}F_{i-1}^jr_{i-1}=0$ for $j=1,\dots,m_e$ by \eqref{eq:thm:linesearch:r_opt-perp} and then
	\begin{align*}
		r_{i-1}^{\HH}F_i^jr_i
		&= r_{i-1}^{\HH}[F_{i-1}^j-\delta_{i-1}F_{i-1}^{j-1}\Phi_{i-1}-\delta_{i-1}\Phi_{i-1}F_{i-1}^{j-1}+\delta_{i-1}^2F_{i-1}^{j-2}\Phi_{i-1}^2+\dotsb]r_i
		\\&= r_{i-1}^{\HH}F_{i-1}^jr_i+O(\delta_{i-1})\|r_{i-1}\|\|r_i\|
		= O(\delta_{i-1}^{3/2})\|r_i\|
		,
	\end{align*}
	together with
	\[
		r_{i-1}^{\HH}x_i=r_{i-1}^{\HH}(x_{i-1}+d_{i-1})=r_{i-1}^{\HH}d_{i-1}=O(\delta_{i-1}),
	\]
	we have
	\begin{equation}\label{eq:thm:rate:compare:prf:rS}
		\|r_{i-1}^{\HH}S_i\|
		= \|r_{i-1}^{\HH}(I-x_ix_i^{\HH}\Phi_i-r_ir_i^{\HH}(r_i^{\HH}r_i)^{-1})\begin{bmatrix}
			F_ir_i & \dots & F_i^{m_e}r_i
		\end{bmatrix} \|
		= O(\delta_{i-1})\|r_i\|
		.
	\end{equation}
	Similarly to the proof of \eqref{eq:thm:rate:compare:prf:claim:inv},
	\begin{equation}\label{eq:thm:rate:compare:prf:SFS}
		\|(S_i^{\HH}\check F_iS_i)^{-1}\|=O(\|r_i\|^{-2}) .
	\end{equation}
	Thus, $r_{i-1}^{\HH}S_i(S_i^{\HH}\check F_iS_i)^{-1}S_i^{\HH}r_{i-1}=O(\delta_{i-1}^2)$ and $d_{i-1}^{\HH}F_iT_id_{i-1} =O(\delta_{i-1}^{3/2})$.
	Thus,
	\begin{equation} \label{eq:thm:rate:compare:prf:dFTd}
		d_{i-1}^{\HH}\check F_iT_id_{i-1}=\delta_{i-1}[O(\delta_{i-1}^{1/2})+O(\delta_i)].
	\end{equation}
	Then \eqref{eq:thm:rate:compare:prf:dFd} and \eqref{eq:thm:rate:compare:prf:dFTd}
	give \eqref{eq:thm:rate:compare:prf:claim:kappa4:dd}.
\end{proof}

\begin{proof}[Proof \emph{of \eqref{eq:thm:rate:compare:prf:claim:kappa4:dr}}]
	Since $\check F_ix_{i-1}=\check F_i(x_i-d_{i-1})=-\check F_id_{i-1}$
	and 
	$\check F_i(I-T_i)=(I-T_i^{\HH})\check F_i$,
	\[
		\beta_i=x_{i-1}^{\HH}\check F_i(I-T_i)r_i=d_{i-1}^{\HH}\check F_i(I-T_i)r_i.
	\]
	By \eqref{eq:thm:rate:compare:prf:F},
	\begin{align*}
		d_{i-1}^{\HH}\check F_i(I-T_i)r_i
		&=  d_{i-1}^{\HH}F_i(I-T_i)r_i - d_{i-1}^{\HH}\Phi_ix_ir_i^{\HH}(I-T_i)r_i - d_{i-1}^{\HH}r_ix_i^{\HH}\Phi_i(I-T_i)r_i -\delta_id_{i-1}^{\HH}\Phi_i(I-P_i)(I-T_i)r_i
		\\&=  d_{i-1}^{\HH}F_i(I-T_i)r_i - d_{i-1}^{\HH}\Phi_ix_ir_i^{\HH}r_i - \delta_id_{i-1}^{\HH}\Phi_i(I-P_i)(I-T_i)r_i
		\\&=   d_{i-1}^{\HH}\check F_{i-1}(I-T_i)r_i+d_{i-1}^{\HH}F_iP_{i-1}(I-T_i)r_i- d_{i-1}^{\HH}\Phi_ix_ir_i^{\HH}r_i+O(\delta_i\delta_{i-1}^{1/2})\|r_i\|
		.
	\end{align*}
	By \eqref{eq:thm:linesearch:r}, $d_{i-1}^{\HH}\check F_{i-1}(I-T_i)r_i=(r_i-r_{i-1})^{\HH}(I-T_i)r_i=r_i^{\HH}r_i+r_{i-1}^{\HH}T_ir_i$.
	Note that
	\begin{align*}
		d_{i-1}^{\HH}F_iP_{i-1}(I-T_i)r_i
		&=d_{i-1}^{\HH}F_ix_{i-1}\frac{x_{i-1}^{\HH}\Phi_{i-1}(I-T_i)r_i}{x_{i-1}^{\HH}\Phi_{i-1}x_{i-1}}
		\\&=(x_i-x_{i-1})^{\HH}F_ix_{i-1}\frac{x_{i-1}^{\HH}\Phi_{i-1}(I-T_i)r_i}{x_{i-1}^{\HH}\Phi_{i-1}x_{i-1}}
		\\&=-x_{i-1}^{\HH}F_ix_{i-1}\frac{x_{i-1}^{\HH}\Phi_{i-1}(I-T_i)r_i}{x_{i-1}^{\HH}\Phi_{i-1}x_{i-1}}
		\\&=-x_{i-1}^{\HH}(F_{i-1}-\delta_{i-1}\Phi_{i-1})x_{i-1}\frac{x_{i-1}^{\HH}\Phi_{i-1}(I-T_i)r_i}{x_{i-1}^{\HH}\Phi_{i-1}x_{i-1}}
		\\&= \delta_{i-1}x_{i-1}^{\HH}\Phi_{i-1}(I-T_i)r_i
		\\&= O(\delta_{i-1})\|r_i\|
		.
	\end{align*}
	Thus,
	\[
		\quad d_{i-1}^{\HH}\check F_i(I-T_i)r_i
		= r_{i-1}^{\HH}T_ir_i +(1-d_{i-1}^{\HH}\Phi_ix_i)r_i^{\HH}r_i +[O(\delta_{i-1})+O(\delta_i\delta_{i-1}^{1/2})]\|r_i\|
		.
	\]
	Note that $\delta_{i-1}\sim d_{i-1}^{\HH}\Phi_id_{i-1}\sim d_{i-1}^{\HH}d_{i-1}$
	and then $\|x_{i-1}\|=\|x_i-d_{i-1}\|\le\|x_i\|+O(\delta_{i-1})$ which means $x_{i-1}$ is bounded.
	Also, note that $ r_i^{\HH}r_i=O(\delta_i)$ and $x_{i-1}^{\HH}\Phi_{i-1}d_{i-1}=0$.
	Thus,
	\begin{align*}
		x_i^{\HH}\Phi_id_{i-1}
		&= d_{i-1}^{\HH}\Phi_id_{i-1} +x_{i-1}^{\HH}\Phi_id_{i-1}
		\\&= O(\delta_{i-1})+x_{i-1}^{\HH}\Phi_{i-1}d_{i-1}+(\delta_{i-1}-\delta_i)x_{i-1}^{\HH}F''(\lambda_1)d_{i-1}
		\\&= O(\delta_{i-1})+O(\delta_i\delta_{i-1}^{1/2})
		.
	\end{align*}
	By \eqref{eq:thm:rate:compare:prf:rS} and \eqref{eq:thm:rate:compare:prf:SFS},
	noticing that
	$\|S_i\|
	\le \left\|\begin{bmatrix}
		F_ir_i & \dots & F_i^{m_e}r_i
	\end{bmatrix} \right\|
	= O(1)\|r_i\|
	$,
	we have
	\begin{align*}
		|r_{i-1}^{\HH}T_ir_i|
		&= |r_{i-1}^{\HH}S_i(S_i^{\HH}F_{i+1}S_i)^{-1}S_i^{\HH}F_{i+1}(I-P_i)r_i|
		\\&\le \|r_{i-1}^{\HH}S_i\|\|(S_i^{\HH}F_{i+1}S_i)^{-1}\|\|S_i^{\HH}F_{i+1}(I-P_i)r_i\|
		= O(\delta_{i-1})\|r_i\|.
	\end{align*}
	Then, to sum up,
	we have \eqref{eq:thm:rate:compare:prf:claim:kappa4:dr}.
\end{proof}

\smallskip
\subparagraph{Acknowledgement.}
The authors thank Prof. Ren-Cang Li for his MATLAB code and helpful discussions.

{\small
	\bibliographystyle{plain}
	\bibliography{cgconv}
}

\end{document}